\documentclass[preprint,10pt]{elsarticle}
\usepackage{amssymb}
\usepackage{atbegshi}
\AtBeginDocument{\AtBeginShipoutNext{\AtBeginShipoutDiscard}}
\usepackage[utf8]{inputenc} 
\usepackage{alltt}
\usepackage{amsfonts}
\usepackage{amsmath}
\usepackage{times}
\usepackage[T1]{fontenc}
\usepackage[english,french]{babel}
\usepackage{lmodern}
\usepackage[a4paper]{geometry}
\usepackage{epstopdf}
\usepackage{mathrsfs}
\usepackage{latexsym,amssymb,amsfonts}
\usepackage{enumerate}
\usepackage{relsize,exscale}
\usepackage{dsfont}
\usepackage{microtype}
\usepackage{xcolor}
\usepackage{enumitem} 
\usepackage{graphicx}
\usepackage{float}
\usepackage{upgreek}
\usepackage{caption}
\usepackage{subfigure}

\usepackage[labelfont=bf]{caption}
\usepackage{hyperref}
\usepackage[noabbrev]{cleveref}
\usepackage{longtable}
\hypersetup{
    colorlinks,
    citecolor=red,
    filecolor=red,
    linkcolor=red,
    urlcolor=red
}

\numberwithin{equation}{section} 

\usepackage{amsthm}
\newtheorem{theo}{Theorem}[section]

\newtheorem{lemm}{Lemma}[section]
\newtheorem{defi}{Definition}[section]

\newtheorem{rema}{Remark}[section]
\theoremstyle{definition}

\usepackage{indentfirst}
\usepackage[toc,page]{appendix}
\journal{*}
\usepackage{geometry}
\usepackage{undertilde} 
\usepackage{cellspace}
\begin{document}
\selectlanguage{english}
\begin{frontmatter}
\title{A novel mathematical analysis and threshold reinforcement of a stochastic dengue epidemic model with L\'{e}vy jumps.}
\author{Driss Kiouach\footnote{Corresponding author.\\
E-mail addresses: \href{d.kiouach@uiz.ac.ma}{d.kiouach@uiz.ac.ma} (Driss Kiouach), \href{salim.elazamielidrissi@usmba.ac.ma}{salim.elazamielidrissi@usmba.ac.ma} (Salim El Azami El-idrissi),\\
\hspace*{2.2cm} \href{yassine.sabbar@usmba.ac.ma}{yassine.sabbar@usmba.ac.ma} (Yassine Sabbar).}, Salim El Azami El-idrissi and Yassine Sabbar}
\address{LPAIS Laboratory, Faculty of Sciences Dhar El Mahraz, Sidi Mohamed Ben Abdellah University, Fez, Morocco.}   
\vspace*{1cm}
\begin{abstract}
 The rampant phenomenon of overpopulation and the remarkable increase of human movements over the last decade have caused an aggressive re-emergence of dengue fever, which made it the subject of several research fields. In this regard, mathematical modeling, and notably through compartmental systems, is considered as an eminent tool to obtain a clear overview of this disease's prevalence behavior. In reality, and like all epidemics, the dengue spread phenomenon is often subject to some randomness due to the different natural environment fluctuations. For this reason, a mathematical formulation that considers suitably as much as possible the external stochasticity is indeed required. By this token, we strive in this work to present and analyze a generalized stochastic dengue model that incorporates both slight and huge environmental perturbations. More precisely, our proposed model is represented under the form of an It\^{o}-L\'{e}vy stochastic differential equations system that we demonstrate its mathematical well-posedness and biological significance. Based on some novel analytical techniques, we prove, and under appropriate hypothetical frameworks, of course, two important asymptotic properties, namely: extinction and persistence in the mean. The theoretical findings show that the dynamics of our disturbed dengue model are mainly determined by the parameters that are narrowly related to the small perturbations' intensities and the jumps magnitudes. In the end, we give certain numerical illustrative examples to support our theoretical findings and to highlight the effect of the adopted mathematical techniques on the results.\\[2mm]  
\hspace*{2.1cm}Persistence in the mean.\\[3pt] 
\end{abstract}
\end{frontmatter}

\section{Introduction and model formulation}
Since ancient times, mankind has had to deal with various very dangerous epidemics which are characterized by rapid spread and  high death rate \cite{hays2005epidemics,dobson}. Often caused by some kind of bacteria or viruses unknown in their time, these epidemics killed millions of people, and thus marked the history of several countries, societies and even Humanity in general \cite{snowden2019epidemics}. By looking a little into the past, we can find in this context many examples such as smallpox, tuberculosis,  plague, cholera, typhus, the "Spanish flu" of 1918, and closer to us SARS, Ebola, Zika virus, HIV and most recently COVID-19. Seemingly, the list of all these diseases is very long, and we cannot write it all down here, but what we can guarantee is that the name of  Dengue fever will undoubtedly appear in it \cite{halstead2007dengue}. 
The dengue fever is a mosquito-borne viral infection caused mainly by one of the four dengue virus stereotypes (DENV-1 to DENV-4). According to the  World Health Organization (WHO) \cite{who}, a significant number of dengue infections produce just mild illness, but many others can lead to an acute flu-like illness which later turns into a potentially fatal complication named severe dengue. With thousands of mortalities and nearly four hundred million infections annually around the globe, this disease is considered to be the deadliest vector-borne epidemic after Malaria \cite{bhatt}. Despite the existence of some suggested developments regarding a remedy for the dengue virus \cite{who2}, until now, no effective vaccine or treatment against it are available in the market \cite{khan2021dengue}. So, early-stage detection  and access to appropriate medical care still the only possible solutions at hand to face this murderous disease. At present, more than one hundred countries are under the threat of dengue fever \cite{brady2012refining}, and what makes matters worse is the ability of this epidemic to affect almost all age groups, that is why a good comprehension of its evolution dynamics is firmly required. In this vein, mathematical modeling can be presented as the most useful, efficient and applicable tool for appropriately describing the dengue fever prevalence and perceiving its effects on a host population, especially in the long term.
\par In order to comprehend and supervise the running of dengue infection, a considerable number of mathematical models, notably compartmental ones, have been suggested and treated in details by several works \cite{newton1992model,focks1993dynamic,focks1993dynamic2}. The first attempt to describe the dengue spread was introduced by Newton and Reiter \cite{newton1992model} in the form of an SEIR model that does not take into consideration the mosquitoes populations. Later, and in the same context, Focks et al. \cite{focks1993dynamic,focks1993dynamic2} closed this loophole by using dynamic table models to illustrate the evolution of these populations. Because of its continued re-emergence \cite{morens2013dengue}, the study of the dengue's spread has not stopped at this stage, and even it remains until now an active research subject that inspires several recent papers, see \cite{agusto2018optimal, wang2019dynamics, champagne2019comparison, zhu2019effects} and the references given there. For example, in \cite{agusto2018optimal} the authors constructed a deterministic dengue's propagation model and parametrized it by employing real data from the 2017 dengue outbreak in Pakistan. In \cite{wang2019dynamics}, Wang and Zhao studied the vaccination effect on the prevalence of dengue fever under the framework of coinfection with Zika virus. A brief discussion on the dengue modeling in both deterministic and stochastic levels is presented in \cite{champagne2019comparison}. The dengue dissemination dynamics with the mosquitoes control, temperature elevation and human mobility restriction are detailedly treated in \cite{zhu2019effects}. In \cite{cai2009global}, Cai et al. drew up a dengue epidemic model with bilinear saturated incidence before going to investigate the global stability of the disease-free and the endemic equilibria. The formulation of their model can be presented by the following ordinary differential equations system:\vspace{-2pt}  
\begin{equation}\label{detr}
\left\lbrace
 \begin{aligned}
\dfrac{\mathrm{d}S}{\text{d}t}&=\Lambda-\dfrac{b \beta S \widehat{I}}{1+a\widehat{I}}-\mu S,\\
\dfrac{\mathrm{d}I}{\text{d}t}&=\dfrac{b \beta S \widehat{I}}{1+a\widehat{I}}-\left(\mu+\rho_0+r_1+r_2\right)I,\\
\dfrac{\mathrm{d}R}{\text{d}t}&=\left(r_1+r_2\right)I-\mu R,\\
\dfrac{\mathrm{d}\widehat{S}}{\text{d}t}&=\hat{\Lambda}-b\hat{\beta} \widehat{S} I -\hat{\mu}\widehat{S},\\
\dfrac{\mathrm{d}\widehat{I}}{\text{d}t}&=b\hat{\beta} \widehat{S} I -\hat{\mu}\widehat{I},
\end{aligned}
\right.
\end{equation}\vspace{-2pt} 
with initial conditions $S(0)>0,I(0)>0,R(0)>0,\widehat{S}(0)>0$ and $\widehat{I}(0)>0$. Clearly, the previous system describes the simultaneous evolution of two populations, the first is the host population denoted by $N(t)$ at time $t$ and subdivided into three compartments of susceptible, infected, and recovered individuals, with densities indicated respectively by $S(t)$, $I(t)$ and $R(t)$. The second population is the mosquitoes one which is denoted at time $t$ by $\widehat{N}(t)$ and partitioned in turn into two classes of susceptible and infected individuals with densities $\widehat{S}(t)$ and $\widehat{I}(t)$ respectively. The recovered mosquitoes class is not considered in this model because of the short lifespan of these insects and the relatively long time required to recover from the dengue disease \cite{cai2009global}. The eleven parameters appearing in system \eqref{detr} are  summarized in the following list:
\begin{itemize}
\item[$\bullet$] $\Lambda$ and $\hat{\Lambda}$ are  the susceptible humans and mosquitoes recruitment rates respectively.
\item[$\bullet$] $\beta$ is the transmission rate from mosquitoes to humans. 
\item[$\bullet$] $a$ is a parameter measuring the inhibitory effect due to psychological and behavioral changes in the susceptible human population. 
\item[$\bullet$] $b$ is the biting rate of mosquitoes, in other words, the average number of bites per mosquito in one day.
\item[$\bullet$] $\rho_0$ is the disease-induced death rate for human individuals. 
\item[$\bullet$] $r_1$ and $r_2$ are respectively the human individuals recovery and  treatment rates. For convenience in writing we denote from now on $r_1+r_2$ by $\rho_1$. 
\item[$\bullet$] $\hat{\beta}$ is the transmission rate from humans to mosquitoes.
\item[$\bullet$] $\mu$ and $\hat{\mu}$ are, in this order, the humans and mosquitoes natural death rates.
\end{itemize}
All parameters listed above are assumed to be in the positive real axis $\mathbb{R}_+:=\left\lbrace x\in\mathbb{R}\mid x>0\right\rbrace$ except $a$ that is supposed  to be just nonnegative. Since $R(t)$ does not appear in the other equations of system \eqref{detr},  the dynamical behavior of dengue infection can be deduced just from the following simplified model:
\begin{equation}\label{detr2}
\left\lbrace
 \begin{aligned}
\dfrac{\mathrm{d}S}{\text{d}t}&= \Lambda-b\beta\dfrac{S\widehat{I}}{1+a\widehat{I}}-\mu S,\\
\dfrac{\mathrm{d}I}{\text{d}t}&= b\beta\dfrac{S\widehat{I}}{1+a\widehat{I}}-\left(\mu+\rho_0+\rho_1\right)I,\\
\dfrac{\mathrm{d}\widehat{S}}{\text{d}t}&= \hat{\Lambda}-b\hat{\beta} \widehat{S} I -\hat{\mu}\widehat{S},\\
\dfrac{\mathrm{d}\widehat{I}}{\text{d}t}&= b\hat{\beta} \widehat{S} I -\hat{\mu}\widehat{I}.
\end{aligned}
\right.
\end{equation} 
According to the mathematical analyses effected in \cite{cai2009global}, the spread dynamics of the aforementioned  dengue model is completely determined by  the basic reproduction number which is expressed in this case by $\mathcal{R}_0=\frac{b^2\beta\Lambda\hat{\beta}\hat{\Lambda}}{\mu\left(\mu+\rho_0+\rho_1\right)\hat{\mu}^2}.$   
More precisely, if $\mathcal{R}_0 \leqslant 1$, the system \eqref{detr2} admits  a unique disease-free equilibrium $\mathcal{E}_o=\left(\frac{\Lambda}{\mu},0,0,0\right)$, and it is globally asymptotically stable in the invariant region $\mathcal{D}=\left\lbrace\left(S,I,\widehat{S},\widehat{I}\right)\in\mathbb{R}_+ \mid S+I\leqslant\frac{\Lambda}{\mu}~\text{and}~\widehat{S}+\widehat{I}\leqslant\frac{\hat{\Lambda}}{\hat{\mu}} \right\rbrace$. On the other hand if $\mathcal{R}_0>1$, the disease-free equilibrium is still present, but it becomes unstable. In this case, the system  \eqref{detr2} admits another steady state $\mathcal{E}_\star=\left(S_\star,I_\star,\widehat{S}_\star,\widehat{I}_\star\right)$ called the endemic equilibrium, and it is globally asymptotically stable in the interior of $\mathcal{D}$. The endemic equilibrium components that we have just mentioned are given as follows: $S_\star=\frac{\hat{\mu}^2(\mu+\rho_0+\rho_1)(1+a\widehat{I}_\star)}{b^2 \beta \hat{\beta}(\Lambda-\hat{\mu}\widehat{I}_\star)},~I_\star=\frac{\hat{\mu}^2\widehat{I}_\star}{b \hat{\beta}(\Lambda-\hat{\mu}\widehat{I}_\star)},~\widehat{S}_\star=\frac{\Lambda}{\mu}-\widehat{I}_\star$ and $\widehat{I}_\star=\frac{\mu\hat{\mu}\left(\mu+\rho_0+\rho_1\right)(\mathcal{R}_0-1)}{\hat{\mu}\left(\mu a+b\beta\right)\left(\mu+\rho_0+\rho_1\right)+b^2\beta\hat{\beta}\hat{\Lambda}}.$\vspace*{4pt}
\par It is well known that environmental uncertainties influence the spread of any disease and make the divination of its behavior  more difficult \cite{lahrouz2011global}. In  such a situation  deterministic models, and despite their ability in providing highly informative results and predictions,  are not suitable enough. Therefore, a new or modified mathematical formulation that takes into account this effect of randomness is really required, especially in the framework of dengue fever dissemination analysis. In this regard,  several authors have proposed and developed many stochastic models that are presenting the dengue infection dynamics from different viewpoints and perspectives.  For example, in \cite{otero2010stochastic} Otero and Solari investigated a probabilistic model depicting the evolution of dengue fever within a spatially fixed humans population, that is to say that mosquitoes dispersal is the only one responsible for the disease's spreading. In order to highlight the role of human mobility on dengue prevalence, Barmak et al. \cite{barmak2016modelling} treated a perturbed model incorporating the infected mosquitoes and humans dynamics with different human mobility patterns. In \cite{liu2018stationary}, Liu at al. took into account the randomness's effect on the dengue fever pervasion by assuming that the solution of the model \eqref{detr2} fluctuates normally around its value. Following this point of view, they explored a  disturbed  version of the dengue epidemic model \eqref{detr2}, in which the environmental perturbations are in the form of proportional whites noises to the variables. Hence, the stochastic dengue model that they studied is presented as follows:  
\begin{equation}\label{stob}
\left\lbrace
 \begin{aligned}
	\mathrm{d}S(t)&=\left[\Lambda-\dfrac{b \beta S(t) \widehat{I}(t)}{1+a\widehat{I}(t)}-\mu S(t)\right]\text{d}t+\sigma_1S(t)\text{d}\mathcal{B}_1(t),\\
\mathrm{d}I(t)&=\left[\dfrac{b \beta S(t) \widehat{I}(t)}{1+a\widehat{I}(t)}-\left(\mu+\rho_0+\rho_1\right)I(t)\right]\text{d}t+\sigma_2I(t)\text{d}\mathcal{B}_2(t),\\
\mathrm{d}\widehat{S}(t)&=\left[\hat{\Lambda}-b\hat{\beta} \widehat{S}(t) I(t)-\hat{\mu}\widehat{S}(t)\right]\text{d}t+\sigma_3\widehat{S}(t)\text{d}\mathcal{B}_3(t),\\
\mathrm{d}\widehat{I}(t)&=\left[b\hat{\beta} \widehat{S}(t)I(t)-\hat{\mu}\widehat{I}(t)\right]\text{d}t+\sigma_4\widehat{I}(t)\text{d}\mathcal{B}_4(t),
\end{aligned}
\right.
\end{equation}
where, the nonnegative constants $\sigma_i~(i=1,2,3,4)$ denote the intensities of the mutually independent Brownian motions $\mathcal{B}_i~(i=1,2,3,4)$. These latter, and even all the stochastic processes or random variables that will be met in this paper, are supposed to be defined on a probability space 
$\left(\Omega,\mathcal{F},\mathbb{P}\right)$ equipped with a filtration $(\mathcal{F}_t)_{t\geqslant 0}$ that satisfies the usual conditions (it is increasing and right continuous while $\mathcal{F}_0$ contains all $\mathbb{P}$-null sets).
\par  The insertion of white noises is a very reasonable and prominent approach to model phenomena in which a given quantity is constantly subjected to slight variable fluctuations, for example, those of the environment during a certain disease's spread \cite{zhang2021threshold,shaikhet2020stability,kiouach2021longtime}. But unfortunately, this method is neither appropriate nor sufficient to represent the effect of strong and sudden external disturbances like climate changes, floods, earthquakes, tornadoes, etc \cite{zhou2016threshold,kiouach2020new}. For this reason, we will make recourse to the well known L\'{e}vy processes which are able to formulate adequately this kind of cases. By taking into account this type of random perturbations, we can extend the model \eqref{stob} to the following system of stochastic differential equations with L\'{e}vy jumps (SDELJ for brevity):
\begin{equation}\label{sto0}
\left\lbrace
 \begin{aligned}
	\mathrm{d}S(t)&=\left[\Lambda- \dfrac{b \beta S(t) \widehat{I}(t)}{1+a\widehat{I}(t)}-\mu S(t)\right]\text{d}t+\sigma_1S(t)\text{d}\mathcal{B}_1(t)+\int_{\mathbb{U}}\xi_1(u)S(t^-)\widetilde{\mathcal{N}}(\text{d}t,\text{d}u),\\
\mathrm{d}I(t)&=\left[\dfrac{b \beta S(t) \widehat{I}(t)}{1+a\widehat{I}(t)}-\left(\mu+\rho_0+\rho_1\right)I(t)\right]\text{d}t+\sigma_2I(t)\text{d}\mathcal{B}_2(t)+\int_{\mathbb{U}}\xi_2(u)I(t^-)\widetilde{\mathcal{N}}(\text{d}t,\text{d}u),\\
\mathrm{d}\widehat{S}(t)&=\left[\hat{\Lambda}-b\hat{\beta} \widehat{S}(t) I(t)-\hat{\mu}\widehat{S}(t)\right]\text{d}t+\sigma_3\widehat{S}(t)\text{d}\mathcal{B}_3(t)+\int_{\mathbb{U}}\xi_3(u)\widehat{S}(t^-)\widetilde{\mathcal{N}}(\text{d}t,\text{d}u),\\
\mathrm{d}\widehat{I}(t)&=\left[b\hat{\beta} \widehat{S}(t)I(t)-\hat{\mu}\widehat{I}(t)\right]\text{d}t+\sigma_4\widehat{I}(t)\text{d}\mathcal{B}_4(t)+\int_{\mathbb{U}}\xi_4(u)\widehat{I}(t^-)\widetilde{\mathcal{N}}(\text{d}t,\text{d}u).
\end{aligned}
\right.
\end{equation}
Here and thereafter, $S(t^-), I(t^-),\widehat{S}(t^-)$ and $\widehat{I}(t^-)$ are respectively standing for the left limits of  $S(t), I(t),\widehat{S}(t)$ and $\widehat{I}(t)$. $\mathcal{N}$ is a Poisson counting measure independent of $\mathcal{B}_i~(i=1,2,3,4)$ with compensating martingale $\widetilde{\mathcal{N}}$ and finite characteristic measure $\nu$ that is defined on a measurable set $\mathbb{U}\subset\mathbb{R}_+$. Also, it is assumed that $\nu$ is a Lévy measure such that $\widetilde{\mathcal{N}}(\textcolor{black}{\textup{d}t},\textup{d}u)=\mathcal{N}(\textcolor{black}{\textup{d}t},\textup{d}u)-\nu(\textup{d}u)\textcolor{black}{\textup{d}t}$ and we suppose that the jumps intensities $\xi_i: \mathcal{U}\to\mathbb{R}$ are continuous functions on $\mathcal{U}$. In order to provide an additional degree of realism to our study, we will suppose that the contact between humans and mosquitoes  is homogeneously mixed. In other words, and in analogy with a chemical reaction, we will adopt the mass  action rates $\beta S\widehat{I}$ and $\beta\widehat{S}I$ as the incidence functions of our epidemic model. The last fact will lead us to the following system  which is none other than system \eqref{sto0} but in the case of the inhibitory effect absence:
\begin{equation}\label{sto1}
\left\lbrace
 \begin{aligned}
	\mathrm{d}S(t)&=\left[\Lambda- b \beta S(t) \widehat{I}(t)-\mu S(t)\right]\text{d}t+\sigma_1S(t)\text{d}\mathcal{B}_1(t)+\int_{\mathbb{U}}\xi_1(u)S(t^-)\widetilde{\mathcal{N}}(\text{d}t,\text{d}u),\\
\mathrm{d}I(t)&=\left[b \beta S(t) \widehat{I}(t)-\left(\mu+\rho_0+\rho_1\right)I(t)\right]\text{d}t+\sigma_2I(t)\text{d}\mathcal{B}_2(t)+\int_{\mathbb{U}}\xi_2(u)I(t^-)\widetilde{\mathcal{N}}(\text{d}t,\text{d}u),\\
\mathrm{d}\widehat{S}(t)&=\left[\hat{\Lambda}-b\hat{\beta} \widehat{S}(t) I(t)-\hat{\mu}\widehat{S}(t)\right]\text{d}t+\sigma_3\widehat{S}(t)\text{d}\mathcal{B}_3(t)+\int_{\mathbb{U}}\xi_3(u)\widehat{S}(t^-)\widetilde{\mathcal{N}}(\text{d}t,\text{d}u),\\
\mathrm{d}\widehat{I}(t)&=\left[b\hat{\beta} \widehat{S}(t)I(t)-\hat{\mu}\widehat{I}(t)\right]\text{d}t+\sigma_4\widehat{I}(t)\text{d}\mathcal{B}_4(t)+\int_{\mathbb{U}}\xi_4(u)\widehat{I}(t^-)\widetilde{\mathcal{N}}(\text{d}t,\text{d}u).
\end{aligned}
\right.
\end{equation}  
For the reader's convenience, we summarize and illustrate the transmission mechanisms of the aforementioned dengue model with the help of the flowchart depicted in figure \ref{flow}.
\par Our principal objective in this manuscript is to explore sufficient conditions for extinction and persistence in the mean of the dengue model \eqref{sto1}. These two asymptotic properties are considered to be sufficient for having an excellent idea of the future dengue pandemic situation. The originality of our work lies essentially in the method that we adopt to estimate  the limit values of the temporal averages $\dfrac{\int_0^t\Psi(s)\:\textup{d}s}{t}$, $\dfrac{\int_0^t\Psi^2(s)\:\textup{d}s}{t}$, $\dfrac{\int_0^t\widehat{\Psi}(s)\:\textup{d}s}{t}$  and $\dfrac{\int_0^t\widehat{\Psi}^2(s)\:\textup{d}s}{t}$, where $\Psi$ and $\widehat{\Psi}$ are respectively the positive solutions of the following systems: 
\begin{align}\label{aux1}
\begin{cases}\text{d}\Psi(t)=\left[\Lambda-\mu \Psi(t)\right]\text{d}t+\sigma_1\Psi(t)\text{d}\mathcal{B}_1(t)+\displaystyle{\int_{\mathbb{U}}\xi_1(u)\Psi(t^-)\widetilde{\mathcal{N}}(\text{d}t,\text{d}u)}, \hspace{0.5cm}\forall t>0,\\
\Psi(0)=S(0)>0.
\end{cases}
\end{align} 
and
\begin{align}\label{aux2}
\begin{cases}\text{d}\widehat{\Psi}(t)=\big[\hat{\Lambda}-\hat{\mu}\widehat{\Psi}(t)\big]\text{d}t+\sigma_3\widehat{\Psi}(t)\text{d}\mathcal{B}_3(t)+\displaystyle{\int_{\mathbb{U}}\xi_3(u)\widehat{\Psi}(t^-)\widetilde{\mathcal{N}}(\text{d}t,\text{d}u)}, \hspace{0.5cm}\forall t>0,\\
\widehat{\Psi}(0)=\widehat{S}(0)>0.
\end{cases}
\end{align} 
Our method permits us to bridge the gap left by the use of classical approaches presented  for example in \cite{zhao2018sharp,zhao2019stochastic}.  Also, we use a new and non-standard analytical technique to obtain a sharper threshold for the extinction case. The performed analysis in this paper seems to be very encouraging to study other related epidemic models and especially those which are perturbed with L\'{e}vy noises.
\begin{figure}[hbtp]
\centering
\includegraphics[scale=0.4]{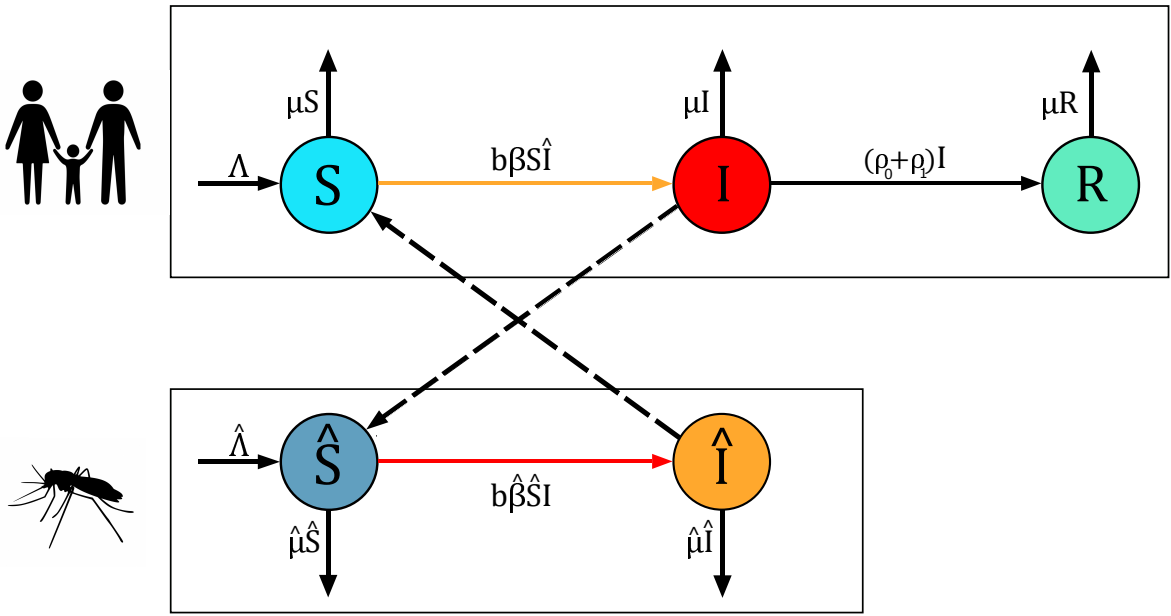}
\caption{Schematic diagram of the proposed dengue compartmental model.}\label{flow}
\end{figure}
\par The remainder of this article is organized as follows: in Section \ref{sec1}, we demonstrate the existence and uniqueness of a positive global-in-time solution to the stochastic dengue model \eqref{sto1}. In Section \ref{sec2}, we provide some sufficient conditions for the dengue disease extinction, whereas those of its persistence in the mean are presented in Section \ref{sec3}. In Section \ref{sec4}, we support our theoretical results with the help of some numerical simulations before drawing the main conclusions of the article in Section \ref{sec5}.
\section{Existence, uniqueness and positivity of the global-in-time solution}\label{sec1}
The first step in exploring the dynamical characteristics of a mathematical population system is to know if it is well-posed or not, where  the well-posedness here designates that the system admits a unique and positive global-in-time solution. In what follows, we will provide some conditions and assumptions under which the well-posedness of the dengue disease model \eqref{sto1} is ensured. But before doing so, let us first introduce the following hypotheses:\vspace*{5pt}
\begin{itemize}
\item[$\bullet$] $\mathbf{(A_1)}:$ The jumps coefficients $(\xi_i)_{1\leqslant i \leqslant 4}$ verify $\mathfrak{M}_1:=\displaystyle{\max_{1\leqslant i \leqslant 4}\left(\int_{\mathbb{U}}\xi_i^2(u)\nu(\text{d}u)\right)<\infty}$.
\item[$\bullet$] $\mathbf{(A_2)}:$ For any $i\in\left\{1,2,3,4\right\}$, $\xi_i(u)>-1$ and $\mathfrak{M}_2:=\displaystyle{\max_{1\leqslant i \leqslant 4}\left(\int_{\mathbb{U}}\Big(\xi_i(u)-\ln\left(1+\xi_i(u)\right)\Big)\nu(\text{d}u)\right)<\infty}$.\vspace*{5pt}
\end{itemize}
\begin{theo}\label{th0}
Let assumptions $\mathbf{(A_1)}$ and $\mathbf{(A_2)}$ hold. Then, for any initial data $\big(S(0),I(0),\widehat{S}(0),\widehat{I}(0)\big)$ belonging to the positive cone $\mathbb{R}_+^4$, there corresponds one and only one solution of the stochastic differential system \eqref{sto1}  on $t\geqslant 0$. Moreover, for all $t\geqslant 0$, this solution will stay in $\mathbb{R}_+^4$ almost surely (a.s. for short).
\end{theo}
\begin{proof}
As supposed in the statement of the theorem, let the hypotheses $(A_1)$ and $(A_2)$ hold. From $(A_1)$, we can easily observe that the coefficients of the system \eqref{sto1} are locally Lipschitz continuous. Hence, by using Theorem 1.19 of \cite{oksendal2007applied}, one can immediately conclude that for any given initial value $\big(S(0),I(0),\widehat{S}(0),\widehat{I}(0)\big)\in\mathbb{R}_+^4$ there corresponds a unique maximal local solution $\big(S(t),I(t),\widehat{S}(t),\widehat{I}(t)\big)$ of \eqref{sto1} on an interval $\left[0,\uptau_e\right)$, where $\uptau_e$ is the explosion time \cite{mao2007stochastic}. At this point, our objective will be to show the globality in time of this solution, in other words,  $\uptau_e=\infty$ almost surely. To this end, let $k_0$ be  a sufficiently large positive integer such that $\big(S(0),I(0),\widehat{S}(0),\widehat{I}(0)\big)\in\left(\dfrac{1}{k_0},k_0\right)$, and consider for any integer $k\geqslant k_0$  the following quantity, which is well defined by the adoption of the convention $\inf \emptyset=\infty$:
\begin{align}\label{tauk}
\uptau_k &=\inf\left\lbrace t\in\left[0,\uptau_e\right)\mid \big(S(t),I(t),\widehat{S}(t),\widehat{I}(t)\big)\not\in \left(\dfrac{1}{k},k\right)^4\right\rbrace\in \mathbb{R}_+\cup \left\lbrace \infty \right\rbrace.
\end{align}
Clearly,  $\left(\dfrac{1}{k},k\right)^4$ is an open subset of $\mathbb{R}^4$, so it follows from Theorem 3.1 of \cite{mao2007stochastic} that  $\uptau_k$ is a stopping time for all $k\geqslant k_0$. Set $\uptau_\infty=\lim\limits_{k\to \infty} \uptau_k$, it is obvious that  $\left(\uptau_k\right)_{k \geqslant k_0}$ is increasing; hence, $\lim\limits_{k\to \infty} \uptau_k=\sup\limits_{k\geqslant k_0} \uptau_k$, and by virtue of Lemma 2.11 in \cite{karatzas1998brownian} $\sup\limits_{k\geqslant k_0} \uptau_k$ is a stopping time, then so is $\uptau_\infty$. Evidently, $\uptau_\infty\leqslant \uptau_e$ (see \cite{kiouach2021advanced}  for more details), so $\uptau_e=\infty$ will follow immediately once we show that $\uptau_\infty=\infty$ a.s.,  and this is exactly what we are going to do for accomplishing our proof. Suppose that $\uptau_\infty=\infty$ a.s. is false, then there is necessarily two positive constants $\varepsilon$ and $\eta$ such that \vspace*{3pt}
\begin{equation}\label{1}
\mathbb{P}\left(\uptau_k\leqslant \eta\right)> \varepsilon~~\text{for all}~k\geqslant k_0.
\end{equation}\vspace*{3pt}
Consider the $\mathcal{C}^2$-function $V:\mathbb{R}_+^4\to\left[0,+\infty\right)$ defined by
$$V\left(x_1,x_2,x_3,x_4\right)=\left(x_1-\upalpha-\upalpha\ln\left(\dfrac{x_1}{\upalpha}\right)\right)+\big(x_2-1-\ln(x_2)\big)+\left(x_3-\upalpha-\upalpha\ln\left(\dfrac{x_3}{\upalpha}\right)\right)+\big(x_4-1-\ln(x_4)\big),$$
where  $\upalpha$ is a positive constant to be determined suitably later. The nonnegativity of this function can be observed from the  inequality $a-1-\ln(a)\geqslant 0,~\forall a>0$.
According to the general multi-dimensional It\^{o}'s formula \big(see \cite[page 8]{oksendal2007applied}\big), we have for all $k \geqslant k_0$ and $t\in\left[0,\uptau_k \right)$\\[3pt]
\begin{align*}
\text{d}V\big(S(t),I(t),\widehat{S}(t),\widehat{I}(t)\big)&=\mathcal{L}V\big(S(t),I(t),\widehat{S}(t),\widehat{I}(t)\big)~\text{d}t+\left(S(t)-\upalpha\right)\sigma_1~\mathrm{d}\mathcal{B}_1(t)+\left(I(t)-1\right)\sigma_2~\mathrm{d}\mathcal{B}_2(t)\\[3pt]
&\quad+\left(\widehat{S}(t)\hspace{-2pt}-\hspace{-2pt}\upalpha\right)\sigma_3\mathrm{d}\mathcal{B}_3(t)\hspace{-1pt}+\hspace{-1pt}\left(\widehat{I}(t)\hspace{-2pt}-\hspace{-2pt}1\right)\sigma_4\mathrm{d}\mathcal{B}_4(t)\hspace{-1pt}+\hspace{-2pt}\int_{\mathbb{U}}\Big(\xi_1(u)S(t^-)-\upalpha\ln\left(1+\xi_1(u)\right)\Big)\widetilde{\mathcal{N}}(\text{d}t,\text{d}u)\\[3pt]
 &\quad+\int_{\mathbb{U}}\Big(\xi_2(u)I(t^-)\hspace{-2pt}-\hspace{-2pt}\ln\left(1+\xi_2(u)\right)\Big)\widetilde{\mathcal{N}}(\text{d}t,\text{d}u)+\hspace{-2pt}\int_{\mathbb{U}}\Big(\xi_3(u)\widehat{S}(t^-)-\upalpha\ln\left(1+\xi_3(u)\right)\Big)\widetilde{\mathcal{N}}(\text{d}t,\text{d}u)\\[3pt]
&\quad+\int_{\mathbb{U}}\Big(\xi_4(u)\widehat{I}(t^-)-\ln\left(1+\xi_4(u)\right)\Big)\widetilde{\mathcal{N}}(\text{d}t,\text{d}u),
\end{align*}
where $\mathcal{L}V\big(S(t),I(t),\widehat{S}(t),\widehat{I}(t)\big)$ is given  by \vspace*{8pt}
\begin{align*}
\mathcal{L}V\big(S(t),I(t),\widehat{S}(t),\widehat{I}(t)\big)&=\left(1-\dfrac{\upalpha}{S(t)}\right)\hspace{-2pt}\times\hspace{-2pt}\left(\Lambda-b \beta S(t) \widehat{I}(t)-\mu S(t)\right)+\left(1-\dfrac{1}{I(t)}\right)\hspace{-2pt}\times\hspace{-2pt}\left(b \beta S(t) \widehat{I}(t)-\left(\mu+\rho_0+\rho_1\right)I(t)\right)\\[3pt]
&\quad+\left(1-\dfrac{\upalpha}{\widehat{S}(t)}\right)\left(\hat{\Lambda}-b\hat{\beta} \widehat{S}(t) I(t)-\hat{\mu}\widehat{S}(t)\right)+\left(1-\dfrac{1}{\widehat{I}(t)}\right)\left(b\hat{\beta} \widehat{S}(t)I(t)-\hat{\mu}\widehat{I}(t)\right)\\[3pt]
&\quad +\dfrac{1}{2}\left(\upalpha\sigma_1^2+\sigma_2^2+\upalpha\sigma_3^2+\sigma_4^2\right)+\upalpha\int_{\mathbb{U}}\Big(\xi_1(u)-\ln\left(1+\xi_1(u)\right)\Big)\nu(\text{d}u)\\[3pt]
&\quad+\int_{\mathbb{U}}\Big(\xi_2(u)-\ln\left(1+\xi_2(u)\Big)\right)\nu(\text{d}u)+\upalpha\int_{\mathbb{U}}\Big(\xi_3(u)-\ln\left(1+\xi_3(u)\right)\Big)\nu(\text{d}u)\\[3pt]
&\quad+\int_{\mathbb{U}}\Big(\xi_4(u)-\ln\left(1+\xi_4(u)\right)\Big)\nu(\text{d}u).
\end{align*}
Therefore
\begin{align*}
\mathcal{L}V\big(S(t),I(t),\widehat{S}(t),\widehat{I}(t)\big)&=\left(\Lambda-\mu S(t)-\left(\mu+\rho_0+\rho_1\right)I(t)+\hat{\Lambda}-\hat{\mu}\widehat{I}(t)-\hat{\mu}\widehat{S}(t)\right)-\dfrac{\upalpha\Lambda}{S(t)}+\upalpha b\beta\widehat{I}(t)+\upalpha\mu\\
&\quad-b\beta S(t) \dfrac{\widehat{I}(t)}{I(t)}+\left(\mu+\rho_0+\rho_1\right)-\dfrac{\upalpha\hat{\Lambda}}{\widehat{S}(t)}+\upalpha b\hat{\beta}I(t)+\upalpha\hat{\mu}-b\dfrac{\hat{\beta}\widehat{S}(t)I(t)}{\widehat{I}(t)}+\hat{\mu}+\dfrac{1}{2}\big(\upalpha\sigma_1^2\\
&\quad+\sigma_2^2+\upalpha\sigma_3^2+\sigma_4^2\big)+\upalpha\int_{\mathbb{U}}\Big(\xi_1(u)-\ln\left(1+\xi_1(u)\right)\Big)\nu(\text{d}u)+\hspace*{-2pt}\int_{\mathbb{U}}\Big(\xi_2(u)\hspace*{-1pt}-\hspace*{-1pt}\ln\left(1+\xi_2(u)\right)\Big)\nu(\text{d}u)\\
&\quad+\upalpha\int_{\mathbb{U}}\Big(\xi_3(u)-\ln\left(1+\xi_3(u)\right)\Big)\nu(\text{d}u)+\int_{\mathbb{U}}\Big(\xi_4(u)-\ln\left(1+\xi_4(u)\right)\Big)\nu(\text{d}u)\\
&\leqslant \left(\upalpha b\hat{\beta}-\left(\mu+\rho_0+\rho_1\right)\right)I(t)\hspace{-0.5pt}+\hspace{-0.5pt}\left(\upalpha b\beta-\hat{\mu}\right)\widehat{I}(t)+\big(\Lambda+\hat{\Lambda}+\upalpha\mu+\left(\mu+\rho_0+\rho_1\right)+\upalpha\hat{\mu}+\hat{\mu}\big)\\
&\quad +\dfrac{1}{2}\left(\upalpha\sigma_1^2+\sigma_2^2+\upalpha\sigma_3^2+\sigma_4^2\right)+\upalpha\int_{\mathbb{U}}\Big(\xi_1(u)-\ln\left(1+\xi_1(u)\right)\Big)\nu(\text{d}u)\\
&\quad+\int_{\mathbb{U}}\Big(\xi_2(u)-\ln\left(1+\xi_2(u)\right)\Big)\nu(\text{d}u)+\upalpha\int_{\mathbb{U}}\Big(\xi_3(u)-\ln\left(1+\xi_3(u)\right)\Big)\nu(\text{d}u)\\
&\quad+\int_{\mathbb{U}}\Big(\xi_4(u)-\ln\left(1+\xi_4(u)\right)\Big)\nu(\text{d}u).
\end{align*}
By choosing $\upalpha=\dfrac{\hat{\mu}}{b\beta}\wedge\dfrac{\mu+\rho_0+\rho_1}{b\hat{\beta}}$, we obtain
\begin{align*}
\mathcal{L}V\big(S(t),I(t),\widehat{S}(t),\widehat{I}(t)\big)&\leqslant \mathfrak{A},
\end{align*}
and $\mathfrak{A}$ here is to the positive constant given by 
\begin{align*}
\mathfrak{A}&=\big(\Lambda+\hat{\Lambda}+\upalpha\mu+\left(\mu+\rho_0+\rho_1\right)+\upalpha\hat{\mu}+\hat{\mu}\big)+\dfrac{1}{2}\left(\upalpha\sigma_1^2+\sigma_2^2+\upalpha\sigma_3^2+\sigma_4^2\right)+\upalpha\int_{\mathbb{U}}\Big(\xi_1(u)-\ln\left(1+\xi_1(u)\right)\Big)\nu(\text{d}u)\\
&\quad +\int_{\mathbb{U}}\Big(\xi_2(u)-\ln\left(1+\xi_2(u)\right)\Big)\nu(\text{d}u)+\upalpha\int_{\mathbb{U}}\Big(\xi_3(u)-\ln\left(1+\xi_3(u)\right)\Big)\nu(\text{d}u)+\int_{\mathbb{U}}\Big(\xi_4(u)-\ln\left(1+\xi_4(u)\right)\Big)\nu(\text{d}u).
\end{align*}
The  remainder of the proof runs on the same lines as the demonstration of Theorem 2.1 in \cite{zhu2016analysis}, so we omit it here for the sake  of space. 
\end{proof}
\section{Stochastic extinction of the dengue disease}\label{sec2}
In mathematical epidemiology, our prime concern after proving the well-posedness is to know if the disease will disappear or it will continue to exist. In this section, we will do our utmost to find some conditions for the dengue disease extinction expressed in terms of noises intensities, jumps coefficients and system parameters. For the reader's convenience, the persistence in the mean will be covered and analyzed separately in the next section. For the sake of brevity, we will adopt from now on  the following notations: \vspace*{5pt}
\begin{itemize}
\begin{minipage}[t]{0.4\linewidth}
\item[$\bullet$] $\Sigma:=\max\big\lbrace \sigma_i^2 \mid i\in\{1,2,3,4\} \big\rbrace.$
\item[$\bullet$] $\widetilde{\xi}(u):=\max\big\lbrace \xi_i(u) \mid i\in\{1,2,3,4\} \big\rbrace.$
\item[$\bullet$] $\widetilde{\mathlarger{\mathlarger{\theta}}}_p(u):=\left(1+\widetilde{\xi}(u)\right)^p-p\times\widetilde{\xi}(u)-1.$
\item[$\bullet$] $\mathlarger{\mathlarger{\theta}}_p(u):=\max\left\lbrace\widetilde{\mathlarger{\mathlarger{\theta}}}_p(u),\utilde{\mathlarger{\mathlarger{\theta}}}_p(u)\right\rbrace.$
\end{minipage}
\begin{minipage}[t]{0.4\linewidth}
\item[$\bullet$] $\rho:=\rho_0+\rho_1.$
\item[$\bullet$] $\utilde{\xi}(u):=\min\big\lbrace \xi_i(u) \mid i\in\{1,2,3,4\} \big\rbrace.$
\item[$\bullet$] $\utilde{\mathlarger{\mathlarger{\theta}}}_p(u):=\left(1+\utilde{\xi}(u)\right)^p-p\times\utilde{\xi}(u)-1.$
\item[$\bullet$] $\varrho_p=\displaystyle{\int_{\mathbb{U}}} \mathlarger{\theta}_p(u) \nu(\text{d}u).$
\end{minipage}
\end{itemize}\vspace*{5pt}
Before stating the main result of this section, we must firstly give the following useful lemma: 
\begin{lemm}\label{lem}
Let \hspace{-2pt} $\big(\hspace{-0.5pt}S(t),\hspace{-0.5pt}I(t),\hspace{-0.5pt}\widehat{S}(t),\hspace{-0.5pt}\widehat{I}(t)\big)$\hspace{-2pt} be the\hspace{-1pt} solution\hspace{-0.5pt} of \eqref{sto1} that\hspace{-1pt} starts\hspace{-1pt} from\hspace{-1pt} a\hspace{-1pt} given\hspace{-1pt} initial value $\hspace{-2pt}$ $\big(\hspace{-0.5pt}S_0,\hspace{-0.5pt}I_0,\hspace{-0.5pt}\widehat{S}_0,\hspace{-0.5pt}\widehat{I}_0\big)\in\mathbb{R}_+^4$. Let also $\Psi(t)$ and  $\widehat{\Psi}(t)$ be respectively the solutions of \eqref{aux1} and \eqref{aux2} that satisfy $\Psi(0)=S(0)$ and $\widehat{\Psi}(0)=\widehat{S}(0)$. If the following conditions are satisfied:
\begin{itemize}
\item[$\bullet$] $\mathbf{(A_3)}:$ There exists a real constant $p>2$ such that $\Delta_p:=\mu\wedge\hat{\mu}-\dfrac{p-1}{2}\Sigma-\dfrac{\varrho_p}{p}>0$.
\item[$\bullet$] $\mathbf{(A_4)}:$ $\mathfrak{M}_3:=\displaystyle{\max_{1\leqslant i \leqslant 4}\left(\int_{\mathbb{U}}\Big(\big(1+\xi_i(u)\big)^2-1\Big)^2\nu(\text{d}u)\right)<\infty}.$
\end{itemize}
Then \vspace{4pt}
\begin{enumerate}[label=$(\textup{\alph*})$]
 \item $\lim\limits_{t\to\infty}\dfrac{\Psi(t)}{t}=0,~~\lim\limits_{t\to\infty}\dfrac{\Psi^2(t)}{t}=0,~~\lim\limits_{t\to\infty}\dfrac{S(t)}{t}=0,~~\text{and}~~\lim\limits_{t\to\infty}\dfrac{I(t)}{t}=0~~a.s.$ 
 \item $\lim\limits_{t\to\infty}\dfrac{\widehat{\Psi}(t)}{t}=0,~~\lim\limits_{t\to\infty}\dfrac{\widehat{\Psi}^2(t)}{t}=0,~~\lim\limits_{t\to\infty}\dfrac{\widehat{S}(t)}{t}=0,~\text{and}~\lim\limits_{t\to\infty}\dfrac{\widehat{I}(t)}{t}=0~~a.s.$ 
 \item $\lim\limits_{t\to\infty}\dfrac{\int_0^t\Psi(s)\:\textup{d}\mathcal{B}_1(s)}{t}=0,~~\lim\limits_{t\to\infty}\dfrac{\int_0^t\Psi^2(s)\:\textup{d}\mathcal{B}_1(s)}{t}=0,~~\lim\limits_{t\to\infty}\dfrac{\int_0^t S(s)\:\textup{d}\mathcal{B}_1(s)}{t}=0,~\text{and}~\lim\limits_{t\to\infty}\dfrac{\int_0^t I(s)\:\textup{d}\mathcal{B}_2(s)}{t}=0~a.s.$ 
 \item $\lim\limits_{t\to\infty}\dfrac{\int_0^t\widehat{\Psi}(s)\:\textup{d}\mathcal{B}_3(s)}{t}=0,~~\lim\limits_{t\to\infty}\dfrac{\int_0^t\widehat{\Psi}^2(s)\:\textup{d}\mathcal{B}_3(s)}{t}=0,~~\lim\limits_{t\to\infty}\dfrac{\int_0^t \widehat{S}(s)\:\textup{d}\mathcal{B}_3(s)}{t}=0,~\text{and}~\lim\limits_{t\to\infty}\dfrac{\int_0^t \widehat{I}(s)\:\textup{d}\mathcal{B}_4(s)}{t}=0~a.s.$
  \item $\lim\limits_{t\to\infty}\dfrac{\int_0^t \int_{\mathbb{U}}\xi_1(u)\Psi(s^-)\widetilde{\mathcal{N}}(\textup{d}s,\textup{d}u)}{t}=0~\text{and}~\lim\limits_{t\to\infty}\dfrac{\int_0^t \int_{\mathbb{U}}\big((1+\xi_1(u))^2-1\big)\Psi^2(s^-)\widetilde{\mathcal{N}}(\textup{d}s,\textup{d}u)}{t}=0~~a.s.$
 \item $\lim\limits_{t\to\infty}\dfrac{\int_0^t \int_{\mathbb{U}}\xi_3(u)\widehat{\Psi}(s^-)\widetilde{\mathcal{N}}(\textup{d}s,\textup{d}u)}{t}=0~\text{and}~\lim\limits_{t\to\infty}\dfrac{\int_0^t \int_{\mathbb{U}}\big((1+\xi_3(u))^2-1\big)\widehat{\Psi}^2(s^-)\widetilde{\mathcal{N}}(\textup{d}s,\textup{d}u)}{t}=0~~a.s.$
 \item $\lim\limits_{t\to\infty}\dfrac{\int_0^t \int_{\mathbb{U}}\xi_1(u)S(s^-)\widetilde{\mathcal{N}}(\textup{d}s,\textup{d}u)}{t}=0~\text{and}~\lim\limits_{t\to\infty}\dfrac{\int_0^t \int_{\mathbb{U}}\xi_2(u)I(s^-)\widetilde{\mathcal{N}}(\textup{d}s,\textup{d}u)}{t}=0~~a.s.$
 \item $\lim\limits_{t\to\infty}\dfrac{\int_0^t \int_{\mathbb{U}}\xi_3(u)\widehat{S}(s^-)\widetilde{\mathcal{N}}(\textup{d}s,\textup{d}u)}{t}=0~\text{and}~\lim\limits_{t\to\infty}\dfrac{\int_0^t \int_{\mathbb{U}}\xi_4(u)\widehat{I}(s^-)\widetilde{\mathcal{N}}(\textup{d}s,\textup{d}u)}{t}=0~~a.s.$
\end{enumerate}\vspace{4pt}
\end{lemm}\vspace*{1pt}
\begin{proof}
The proof of this lemma is similar in spirit to that of Lemma 2.5 in \cite{kiouach2020new}. Hereby, it is omitted here. 
\end{proof}\vspace*{1pt}
\begin{lemm}\label{lem2}
Let $\Psi(t)$ and  $\widehat{\Psi}(t)$ be, in this order, the solutions of systems \eqref{aux1} and \eqref{aux2} that start respectively from the  given initial values  $\Psi(0)>0$ and $\widehat{\Psi}(0)>0$. Assume that the conditions $\mathbf{(A_3)}$ and $\mathbf{(A_4)}$ hold. Then,\vspace*{5pt} 
\begin{enumerate}[label=$(\textup{\alph*})$]
\begin{minipage}[h]{0.4\linewidth}
\item $\displaystyle{\Upsilon:=2\mu-\sigma_1^2-\int_{\mathbb{U}}\xi_1^2(u)\nu(\textup{d}u)>0}~~a.s.$\label{lem2a}
\item $\lim\limits_{t\to\infty}\dfrac{\int_0^t\Psi(s)\:\textup{d}s}{t}=\dfrac{\Lambda}{\mu}~~a.s.$\label{lem2b}
\item $\lim\limits_{t\to\infty}\dfrac{\int_0^t\Psi^2(s)\:\textup{d}s}{t}=\dfrac{2\Lambda^2}{\mu\Upsilon}~~a.s.$\label{lem2c}
\end{minipage}
\begin{minipage}[h]{0.4\linewidth}
\item $\displaystyle{\widehat{\Upsilon}:=2\hat{\mu}-\sigma_3^2-\int_{\mathbb{U}}\xi_3^2(u)\nu(\textup{d}u)}>0~~a.s.$\label{lem2d}
\item $\lim\limits_{t\to\infty}\dfrac{\int_0^t \widehat{\Psi}(s)\:\textup{d}s}{t}=\dfrac{\hat{\Lambda}}{\hat{\mu}}~~a.s.$\label{lem2e}
\item $\lim\limits_{t\to\infty}\dfrac{\int_0^t \widehat{\Psi}^2(s)\:\textup{d}s}{t}=\dfrac{2\hat{\Lambda}^2}{\hat{\mu}\widehat{\Upsilon}}~~a.s.$\label{lem2f}
\end{minipage}
\end{enumerate}
\end{lemm}
\begin{proof}
Let $\Psi(0)$ be a positive initial datum, by integrating the both sides of equation \eqref{aux1} from $0$ to $t$, and then dividing by $t$, we get
\begin{align*}
 \frac{\Psi(t)-\Psi(0)}{t}=\Lambda-\mu\times\frac{\int^t_0\Psi(s)\:\textup{d}s}{t}+\sigma_1\times\frac{\int_0^t \Psi(s)\:\textup{d}\mathcal{B}_1(s)}{t} +\dfrac{\int^t_0\int_\mathbb{U} \xi_1(u)\Psi(s^{-})\widetilde{\mathcal{N}}(\textup{d}s,\textup{d}u)}{t}.
\end{align*} 
Therefore
\begin{align*}
\frac{\int^t_0\Psi(s)\:\textup{d}s}{t} =\dfrac{\Lambda}{\mu}-\frac{\Psi(t)-\Psi(0)}{\mu t}+\sigma_1\times\frac{\int_0^t \Psi(s)\:\textup{d}\mathcal{B}_1(s)}{\mu t} +\dfrac{\int^t_0\int_\mathbb{U} \xi_1(u)\Psi(s^{-})\:\widetilde{\mathcal{N}}(\textup{d}s,\textup{d}u)}{\mu t}.
\end{align*} 
Letting $t$ go to infinity in the last equality and then using  Lemma \ref{lem} yields
\begin{equation}\label{1.5}
\lim\limits_{t\to\infty}\dfrac{\int_0^t\Psi(s)\:\textup{d}s}{t}=\dfrac{\Lambda}{\mu}~~a.s.
\end{equation}
Hence, the item \ref{lem2b} of the lemma is proved. We are now in a position to show the statements \ref{lem2a} and \ref{lem2c}. By applying the generalized It\^{o}'s formula \big(see \cite{oksendal2007applied}\big) to $\Psi^2$ we obtain \vspace{3pt}
\begin{align}\label{2}
\textup{d}\Psi^2(t)&=\bigg(2\Psi(t)\Big(\Lambda-\mu \Psi(t)\Big)+\sigma_1^2\Psi^2(t)+\int_{\mathbb{U}}\Psi^2(t)\Big((1+\xi_1(u))^2-1-2\xi_1(u)\Big)\nu(\textup{d}u)\bigg)\textup{d}t+2\sigma_1 \Psi^2(t)\:\textup{d} \mathcal{B}_1(t)\nonumber\\
&\;\;\;+\int_{\mathbb{U}} \Psi^2(t^{-})\Big((1+\xi_1(u))^2-1\Big)\:\widetilde{\mathcal{N}}(\textup{d}t,\textup{d}u).
\end{align}
Integrating  both sides of \eqref{2}  from $0$ to $t$ and then dividing by $t$ gives
\begin{align*}
\dfrac{\Psi^2(t)-\Psi^2(0)}{t}&= 2\Lambda\times\dfrac{\int^t_0\Psi(s)\textup{d}s}{t}-\bigg(2\mu-\sigma_1^2-\int_{\mathbb{U}}\xi_1^2(u)\nu(\textup{d}u)\bigg)\times\dfrac{\int^t_0\Psi^2(s)\textup{d}s}{t}+2\sigma_1\times\dfrac{\int^t_0\Psi^2(s)\textup{d} \mathcal{B}_1(s)}{t}
\\&\;\;\;+\dfrac{\int^t_0\int_\mathbb{U}\Psi^2(s^-)\Big((1+\xi_1(u))^2-1\Big)\widetilde{\mathcal{N}}(\textup{d}s,\textup{d}u)}{t}.
\end{align*}
So
\begin{align}\label{3}
\bigg(2\mu-\sigma_1^2-\int_{\mathbb{U}}\xi_1^2(u)\nu(\textup{d}u)\bigg)\times\dfrac{\int^t_0\Psi^2(s)\textup{d}s}{t}&=\dfrac{\Psi^2(0)-\Psi^2(t)}{t}+2\Lambda\times\dfrac{\int^t_0\Psi(s)\textup{d}s}{t}+2\sigma_1\times\dfrac{\int^t_0\Psi^2(s)\textup{d} \mathcal{B}_1(s)}{t}\nonumber\\
&\;\;\;+\dfrac{\int^t_0\int_\mathbb{U}\Psi^2(s^-)\Big((1+\xi_1(u))^2-1\Big)\widetilde{\mathcal{N}}(\textup{d}s,\textup{d}u)}{t}.
\end{align}
Clearly, the constant $\displaystyle{\Upsilon:=2\mu-\sigma_1^2-\int_{\mathbb{U}}\xi_1^2(u)\nu(\text{d}u)}$ is not zero, since if it is not the case, we will tend $t$ to infinity and come across the contradictory equality $\frac{2\Lambda^2}{\mu}=0$. For this reason, one can easily divide the both sides of \eqref{3} by $\Upsilon$  and deduce that
\begin{align*}
\dfrac{\int^t_0\Psi^2(s)\textup{d}s}{t}&=\dfrac{\Psi^2(0)-\Psi^2(t)}{\Upsilon t}+2\Lambda\times\dfrac{\int^t_0\Psi(s)\textup{d}s}{\Upsilon t}+2\sigma_1\times\dfrac{\int^t_0\Psi^2(s)\textup{d} \mathcal{B}_1(s)}{\Upsilon t}+\dfrac{\int^t_0\int_\mathbb{U}\Psi^2(s^-)\Big((1+\xi_1(u))^2-1\Big)\widetilde{\mathcal{N}}(\textup{d}s,\textup{d}u)}{\Upsilon t}.
\end{align*}
Using Lemma \ref{lem} together with \eqref{1.5} implies that
$$\lim\limits_{t\to\infty}\dfrac{\int_0^t\Psi^2(s)\:\textup{d}s}{t}=\dfrac{2\Lambda^2}{\mu\Upsilon}~~a.s.$$
and by observing the positivity of $\dfrac{\int_0^t\Psi^2(s)\:\textup{d}s}{t}$, we can conclude at the same time that $\Upsilon>0$. Hence, the claims \ref{lem2a} and \ref{lem2c} of the lemma are proved. 
By an analogous argument, the assertions \ref{lem2c}, \ref{lem2d} and \ref{lem2e} can be drawn, and this finishes the proof.     
\end{proof}
\begin{rema}
By making use of the famous stochastic comparison theorem (see \cite{peng2006necessary}), we can easily assert ,and for almost all $w\in \Omega$, that
\begin{align}\label{3.25}
\begin{cases}
S(t,w)\leqslant \Psi(t,w)\\
\widehat{S}(t,w)\leqslant \widehat{\Psi}(t,w)
\end{cases} ~~\text{for all $t>0$.}
\end{align}     
\end{rema}
\begin{rema}
In the l\'{e}vy jumps case, the previous lemma can be seen as an alternative method to overcome the inexistence of an explicit expression for the stationary distribution of \eqref{aux1} and \eqref{aux2}. This problem remains an open question until now, and we can find in the literature  several works (see for example \cite{zhao2018sharp} and \cite{zhao2019stochastic}) that present the threshold analysis of their epidemiological model with a formulation incorporating an unknown stationary distribution. 
\end{rema}
\begin{defi}[Ordinary stochastic extinction \cite{kiouach2021advanced}]
For system \eqref{sto1} the infected individuals $I(t)$ and $\widehat{I}(t)$  are said to be stochastically extinct if $\lim\limits_{t\to\infty} \left(\widehat{I}(t)+I(t)\right)=0~~a.s.$
\end{defi} 
\begin{defi}[Exponential stochastic extinction \cite{kiouach2021longtime}]
The infected individuals $I(t)$ and $\widehat{I}(t)$  appearing in the system  \eqref{sto1} are called exponentially stochastically extinctive if  $$\limsup_{t\to\infty}\dfrac{1}{t}\ln\left(I(t)+\widehat{I}(t)\right)<0~~a.s.$$  
\end{defi} 
\begin{rema}
It is easily seen that the stochastic exponential extinction implies the ordinary stochastical one (see \cite{ji2014threshold}), but the converse is not true in general.
\end{rema}

In order to simplify the writing of the next theorem, we introduce these conventions:\vspace*{10pt}
\begin{itemize}
\item[$\bullet$] For all $x\in\mathbb{R}$,  $ x^+:=\max\{0,x\}=\dfrac{x+\left|x\right|}{2}$ \big(this function is commonly known as the ramp function, see \cite{nair2011advanced}\big). \vspace*{3pt}
\item[$\bullet$] $\underline{\mathfrak{b}}(u):=\big(\xi_2(u)\wedge\xi_4(u)-\ln\left(1+\xi_2(u)\wedge\xi_4(u)\right)\big)\times \mathds{1}_{\left\{\xi_2(u)\wedge\xi_4(u)>0\right\}}.$ \vspace*{4pt}
\item[$\bullet$] $\overline{\mathfrak{b}}(u):=\big(\xi_2(u)\vee\xi_4(u)-\ln\left(1+\xi_2(u)\vee\xi_4(u)\right)\big)\times \mathds{1}_{\left\{\xi_2(u)\vee\xi_4(u)\leqslant 0\right\}}.$ \vspace*{3pt}
\item[$\bullet$] $\mathfrak{B}:=\displaystyle{\int_{\mathbb{U}}\Big(\overline{\mathfrak{b}}(u)+\underline{\mathfrak{b}}(u)\Big)\nu(\textup{d}u)}.$ \vspace*{3pt}
\item[$\bullet$] $\mathfrak{C}:=\dfrac{(\sigma_2\sigma_4)^2}{2(\sigma_2^2+\sigma_4^2)}.$ \vspace*{3pt}
\item[$\bullet$] $\mathfrak{D}:=\left[(\mu+\rho)\vee\hat{\mu}\right]\times \left(\sqrt{\mathcal{R}_0}-1\right)^+-\left[(\mu+\rho)\wedge\hat{\mu}\right]\times \left( 1-\sqrt{\mathcal{R}_0}\right)^+.$
\end{itemize}
\vspace*{10pt}
\begin{theo} \label{th}
Let $\big(\hspace{-0.5pt}S_0,\hspace{-0.5pt}I_0,\hspace{-0.5pt}\widehat{S}_0,\hspace{-0.5pt}\widehat{I}_0\big)\in\mathbb{R}_+^4$, and let $\big(\hspace{-0.5pt}S(t),\hspace{-0.5pt}I(t),\hspace{-0.5pt}\widehat{S}(t),\hspace{-0.5pt}\widehat{I}(t)\big)$ denote the solution of \eqref{sto1} that satisfies the \\[3pt] initial condition $\big(\hspace{-0.5pt}S(0),\hspace{-0.5pt}I(0),\hspace{-0.5pt}\widehat{S}(0),\hspace{-0.5pt}\widehat{I}(0)\big)=\big(\hspace{-0.5pt}S_0,\hspace{-0.5pt}I_0,\hspace{-0.5pt}\widehat{S}_0,\hspace{-0.5pt}\widehat{I}_0\big)$. If hypotheses $\mathbf{(A_1)}-\mathbf{(A_4)}$ hold, and if moreover we have \vspace*{3pt}
\begin{itemize}
\item[$\bullet$]$\mathbf{(A_5):}$ $\mathfrak{M}_4:=\displaystyle{\max_{1\leqslant i \leqslant 4}\left(\int_{\mathbb{U}}\Big(\ln\big(1+\xi_i(u)\big)\Big)^2\nu(\textup{d}u)\right)<\infty}$. 
\end{itemize}\vspace*{3pt}
Then
$$\limsup_{t\to\infty}\dfrac{1}{t}\ln\left(\dfrac{b\hat{\beta}\hat{\Lambda}}{\hat{\mu}^2(\mu+\rho)}I(t)+\frac{\sqrt{\mathcal{R}_0}}{\hat{\mu}}\widehat{I}(t)\right)\leqslant \mathlarger{\mathlarger{\kappa}}~~a.s,$$
where $$\mathlarger{\mathlarger{\kappa}}:=\mathfrak{D}-\mathfrak{C}-\mathfrak{B}+\dfrac{\hat{\mu}\sqrt{\mathcal{R}_0}}{2}\left(\dfrac{2\mu}{\Upsilon}-1\right)^{\frac{1}{2}}+\dfrac{(\mu+\rho)\sqrt{\mathcal{R}_0}}{2}\left(\dfrac{2\hat{\mu}}{\widehat{\Upsilon}}-1\right)^{\frac{1}{2}}.$$ In particular, if the condition $\mathlarger{\mathlarger{\mathlarger{\kappa}}}<0$ is verified, then the disease will die out exponentially almost surely.
\end{theo}
\begin{proof}
First of all, let us define a $\mathcal{C}^2$ function $W:\mathbb{R}_+^2\to \mathbb{R}$ by $$W(x_1,x_2)=\ln\Bigg(\overbrace{\dfrac{b\hat{\beta}\hat{\Lambda}}{\hat{\mu}^2(\mu+\rho)}}^{:=\lambda_1} x_1+\overbrace{\dfrac{\sqrt{\mathcal{R}_0}}{\hat{\mu}}}^{:=\lambda_2}x_2\Bigg).$$
Applying the It\^{o}'s formula to $W\big(I(t),\widehat{I}(t)\big)$ shows that for all $t\geqslant 0$ we have
\begin{align*}
 \textup{d}W\big(I(t),\widehat{I}(t)\big)&= \mathcal{L}W\big(I(t),\widehat{I}(t)\big)\textup{d}t+\dfrac{\left(\lambda_1\sigma_2I(t)\text{d}\mathcal{B}_2(t)+\lambda_2\sigma_4\widehat{I}(t)\text{d}\mathcal{B}_4(t)\right)}{\lambda_1I(t)+\lambda_2\widehat{I}(t)}\\
 &\quad+\mathlarger{\mathlarger{\int_{\mathbb{U}}}}\ln\left(1+\frac{\lambda_1\xi_2(u)I(t)+\lambda_2\xi_4(u)\widehat{I}(t)}{\lambda_1 I(t)+\lambda_2\widehat{I}(t)}\right)\widetilde{\mathcal{N}}(\textup{d}t,\textup{d}u),
\end{align*}
where 
\begin{align}\label{3.5}
\hspace*{-1cm}\mathcal{L}W\big(I(t),\widehat{I}(t)\big)&=\dfrac{\lambda_1\left(b \beta S(t) \widehat{I}(t)-\left(\mu+\rho_0+\rho_1\right)I(t)\right)+\lambda_2\left(b\hat{\beta} \widehat{S}(t)I(t)-\hat{\mu}\widehat{I}(t)\right)}{\lambda_1I(t)+\lambda_2\widehat{I}(t)}-\dfrac{\left(\lambda_1^2\sigma_2^2(I(t))^2+\lambda_2^2\sigma_4^2(\widehat{I}(t))^2\right)}{2\left(\lambda_1I(t)+\lambda_2\widehat{I}(t)\right)^2}\nonumber\\
 &\quad+\mathlarger{\mathlarger{\int_{\mathbb{U}}}}\Bigg(\ln\left(1+\frac{\lambda_1\xi_2(u)I(t)+\lambda_2\xi_4(u)\widehat{I}(t)}{\lambda_1 I(t)+\lambda_2\widehat{I}(t)}\right)-\frac{\lambda_1\xi_2(u)I(t)+\lambda_2\xi_4(u)\widehat{I}(t)}{\lambda_1 I(t)+\lambda_2\widehat{I}(t)}\Bigg)~\nu(\textup{d}u).
\end{align}
By virtue of the well-known Cauchy-Schwartz inequality  (see for example \cite{yin2017new}), we have \vspace*{7pt}
\begin{equation*}
\left(\lambda_1I(t)+\lambda_2\widehat{I}(t)\right)^2=\left(\dfrac{1}{\sigma_2}\lambda_1\sigma_2I(t)+\dfrac{1}{\sigma_4}\lambda_2\sigma_4\widehat{I}(t)\right)^2\leqslant\left(\dfrac{1}{\sigma_2^2}+\dfrac{1}{\sigma_4^2}\right)\left(\lambda_1^2\sigma_2^2(I(t))^2+\lambda_2^2\sigma_4^2(\widehat{I}(t))^2\right).
\end{equation*} \vspace*{5pt}
So
\vspace*{5pt}
\begin{equation}\label{4}
-\dfrac{1}{\left(\lambda_1 I(t)+\lambda_2\widehat{I}(t)\right)^2}\left(\lambda_1^2\sigma_2^2(I(t))^2+\lambda_2^2\sigma_4^2(\widehat{I}(t))^2\right)\leqslant -\dfrac{(\sigma_2\sigma_4)^2}{\sigma_2^2+\sigma_4^2}.
\end{equation}
\vspace*{5pt}
At the same time, it follows from the monotonicity of the function $a\mapsto \ln(1+a)-a$ on $\left(-1,+\infty\right)$ \big(it is increasing over $\left(-1,0\right)$ and decreasing over $\left[0,+\infty\right)$\big) that\vspace*{5pt}
\begin{equation}\label{5}
\mathlarger{\mathlarger{\int_{\mathbb{U}}}}\Bigg(\ln\left(1+\frac{\lambda_1\xi_2(u)I(t)+\lambda_2\xi_4(u)\widehat{I}(t)}{\lambda_1 I(t)+\lambda_2\widehat{I}(t)}\right)-\frac{\lambda_1\xi_2(u)I(t)+\lambda_2\xi_4(u)\widehat{I}(t)}{\lambda_1 I(t)+\lambda_2\widehat{I}(t)}\Bigg)~\nu(\textup{d}u)\leqslant -\mathfrak{B}.
\end{equation}\vspace*{5pt}
Combining \eqref{4}, \eqref{5} and \eqref{3.25} with \eqref{3.5} yields\vspace*{5pt}
\begin{align}\label{5.5}
\mathcal{L}W\big(I(t),\widehat{I}(t)\big)&\leqslant\dfrac{\lambda_1\left(b \beta \Psi(t) \widehat{I}(t)-\left(\mu+\rho_0+\rho_1\right)I(t)\right)+\lambda_2\left(b\hat{\beta} \widehat{\Psi}(t)I(t)-\hat{\mu}\widehat{I}(t)\right)}{\lambda_1I(t)+\lambda_2\widehat{I}(t)}-\mathfrak{C}-\mathfrak{B}\nonumber\\
&\leqslant\dfrac{\lambda_1\left(b \beta \dfrac{\Lambda}{\mu} \widehat{I}(t)-\left(\mu+\rho_0+\rho_1\right)I(t)\right)+\lambda_2\left(b\hat{\beta} \dfrac{\hat{\Lambda}}{\hat{\mu}}I(t)-\hat{\mu}\widehat{I}(t)\right)}{\lambda_1I(t)+\lambda_2\widehat{I}(t)}-\mathfrak{C}-\mathfrak{B}\nonumber\\
&\quad+\dfrac{\lambda_1 b \beta  \widehat{I}(t)}{\lambda_1I(t)+\lambda_2\widehat{I}(t)}\left(\Psi(t)-\dfrac{\Lambda}{\mu}\right)+\dfrac{\lambda_2 b \hat{\beta}I(t)}{\lambda_1I(t)+\lambda_2\widehat{I}(t)}\left(\widehat{\Psi}(t)-\dfrac{\hat{\Lambda}}{\hat{\mu}}\right)\nonumber\\
&\leqslant \dfrac{\Big(\lambda_1 b \beta \dfrac{\Lambda}{\mu}-\lambda_2\hat{\mu}\Big) \widehat{I}(t)+\left(\lambda_2 b\hat{\beta} \dfrac{\hat{\Lambda}}{\hat{\mu}}-\lambda_1\left(\mu+\rho_0+\rho_1\right)\right)I(t)}{\lambda_1I(t)+\lambda_2\widehat{I}(t)}-\mathfrak{C}-\mathfrak{B}\nonumber\\
&\quad+\dfrac{\lambda_1 b \beta  \widehat{I}(t)}{\lambda_1I(t)+\lambda_2\widehat{I}(t)}\left(\Psi(t)-\dfrac{\Lambda}{\mu}\right)^++\dfrac{\lambda_2 b \hat{\beta} I(t)}{\lambda_1I(t)+\lambda_2 \widehat{I}(t)}\left(\widehat{\Psi}(t)-\dfrac{\hat{\Lambda}}{\hat{\mu}}\right)^+.
\end{align}\vspace*{5pt}
Since $\mathcal{R}_0=\lambda_1 b \beta \dfrac{\Lambda}{\mu}$, $\sqrt{\mathcal{R}_0}=\lambda_2\hat{\mu}$ and $\lambda_1(\mu+\rho)=b\hat{\beta} \dfrac{\hat{\Lambda}}{\hat{\mu}^2}$,  we get\vspace*{5pt}
\begin{align*}
\mathcal{L}W\big(I(t),\widehat{I}(t)\big)&\leqslant\dfrac{\left(\mathcal{R}_0-\sqrt{\mathcal{R}_0}\right) \widehat{I}(t)+\Big(\sqrt{\mathcal{R}_0}\lambda_1(\mu+\rho)-\lambda_1\left(\mu+\rho_0+\rho_1\right)\Big)I(t)}{\lambda_1I(t)+\lambda_2\widehat{I}(t)}-\mathfrak{C}-\mathfrak{B}\\
&\quad+\dfrac{\lambda_1 b \beta}{\lambda_2}\left(\Psi(t)-\dfrac{\Lambda}{\mu}\right)^++\dfrac{\lambda_2 b \hat{\beta}}{\lambda_1}\left(\widehat{\Psi}(t)-\dfrac{\hat{\Lambda}}{\hat{\mu}}\right)^+\\
&\leqslant\dfrac{\left(\sqrt{\mathcal{R}_0}-1\right)\left(\lambda_1\left(\mu+\rho_0+\rho_1\right)I(t)+\lambda_2\hat{\mu}\widehat{I}(t)\right)}{\lambda_1I(t)+\lambda_2\widehat{I}(t)}-\mathfrak{C}-\mathfrak{B}+\dfrac{\lambda_1 b \beta}{\lambda_2}\left(\Psi(t)-\dfrac{\Lambda}{\mu}\right)^+\\
&\quad+\dfrac{\lambda_2 b \hat{\beta}}{\lambda_1}\left(\widehat{\Psi}(t)-\dfrac{\hat{\Lambda}}{\hat{\mu}}\right)^+\\
&\leqslant \mathfrak{D}-\mathfrak{C}-\mathfrak{B}+\dfrac{\lambda_1 b \beta}{\lambda_2}\left(\Psi(t)-\dfrac{\Lambda}{\mu}\right)^++\dfrac{\lambda_2 b \hat{\beta}}{\lambda_1}\left(\widehat{\Psi}(t)-\dfrac{\hat{\Lambda}}{\hat{\mu}}\right)^+.
\end{align*}\vspace*{5pt}
Hence, we obtain \vspace*{5pt}
\begin{align*}
\textup{d}W\big(I(t),\widehat{I}(t)\big) &\leqslant \left( \mathfrak{D}\hspace{-1pt}-\hspace{-1pt}\mathfrak{C}\hspace{-1pt}-\hspace{-1pt}\mathfrak{B}+\dfrac{\lambda_1 b \beta}{\lambda_2}\left(\Psi(t)-\dfrac{\Lambda}{\mu}\right)^+\hspace{-5pt}+\dfrac{\lambda_2 b \hat{\beta}}{\lambda_1}\left(\widehat{\Psi}(t)-\dfrac{\hat{\Lambda}}{\hat{\mu}}\right)^+\right)\textup{d}t+\dfrac{\left(\lambda_1\sigma_2I(t)\text{d}\mathcal{B}_2(t)+\lambda_2\sigma_4\widehat{I}(t)\text{d}\mathcal{B}_4(t)\right)}{\lambda_1I(t)+\lambda_2\widehat{I}(t)}\\
 &\quad+\int_{\mathbb{U}}\ln\big(1+\xi_2(u)\vee\xi_4(u)\big)\widetilde{\mathcal{N}}(\textup{d}t,\textup{d}u)
\end{align*}\vspace*{5pt}
 Integrating the last inequality from $0$ to $t$, and then dividing by $t$ on both sides gives\vspace*{5pt}
\begin{align}\label{6}
\dfrac{W\big(I(t),\widehat{I}(t)\big)}{t} &\leqslant \dfrac{W\big(I(0),\widehat{I}(0)\big)}{t}+ \mathfrak{D}-\mathfrak{C}-\mathfrak{B}+\dfrac{\lambda_1 b \beta}{\lambda_2 t}\mathlarger{\int_{0}^t}\left(\Psi(s)-\dfrac{\Lambda}{\mu}\right)^+\textup{d}s+\dfrac{\lambda_2 b \hat{\beta}}{\lambda_1t}\mathlarger{\int_{0}^t}\left(\widehat{\Psi}(s)-\dfrac{\hat{\Lambda}}{\hat{\mu}}\right)^+\textup{d}s\nonumber\\
&\quad+\dfrac{1}{t}\times\underbrace{\left(\mathlarger{\int_{0}^t}\dfrac{\lambda_1\sigma_2I(s)}{\lambda_1I(s)+\lambda_2\widehat{I}(s)}\text{d}\mathcal{B}_2(s)+\mathlarger{\int_{0}^t}\dfrac{\lambda_2\sigma_4\widehat{I}(s)}{\lambda_1I(s)+\lambda_2\widehat{I}(s)}\text{d}\mathcal{B}_4(s)\right)}_{:=\mathbf{M}_1(t)}\nonumber\\
 &\quad+\dfrac{1}{t}\times\underbrace{\int_{0}^t\int_{\mathbb{U}}\ln\big(1+\xi_2(u)\vee\xi_4(u)\big)\widetilde{\mathcal{N}}(\textup{d}s,\textup{d}u)}_{:=\mathbf{M}_2(t)}.
\end{align}\vspace*{5pt}
On the other hand, we can see by employing the classical H\"{o}lder's inequality that\vspace*{5pt}
\begin{align*}
\dfrac{1}{t}\mathlarger{\int_{0}^t}\left(\Psi(s)-\dfrac{\Lambda}{\mu}\right)^+\textup{d}s&=\dfrac{1}{2t}\mathlarger{\int_{0}^t}\left(\Psi(s)-\dfrac{\Lambda}{\mu}\right)~\textup{d}s+\dfrac{1}{2t}\mathlarger{\int_{0}^t}\left|\Psi(s)-\dfrac{\Lambda}{\mu}\right|~\textup{d}s\\
&\leqslant \dfrac{1}{2t}\mathlarger{\int_{0}^t}\left(\Psi(s)-\dfrac{\Lambda}{\mu}\right)~\textup{d}s+\dfrac{1}{2\sqrt{t}}\left(\mathlarger{\int_{0}^t}\left(\Psi(s)-\dfrac{\Lambda}{\mu}\right)^2\textup{d}s\right)^{\frac{1}{2}}\\
&\leqslant \dfrac{1}{2}\left(\dfrac{1}{t}\int_{0}^t\Psi(s)~\textup{d}s-\dfrac{\Lambda}{\mu}\right)+\dfrac{1}{2}\left(\dfrac{1}{t}\mathlarger{\int_{0}^t}\left(\Psi^2(s)-\dfrac{2\Lambda}{\mu}\Psi(s)+\dfrac{\Lambda^2}{\mu^2}\right)\textup{d}s\right)^{\frac{1}{2}}.
\end{align*}
This last fact together with  Lemma \ref{lem2} implies that\vspace*{5pt}
\begin{equation}\label{7}
\lim_{t\to\infty}\dfrac{1}{t}\mathlarger{\int_{0}^t}\left(\Psi(s)-\dfrac{\Lambda}{\mu}\right)^+\textup{d}s \leqslant \frac{1}{2}\left(\dfrac{2\Lambda^2}{\mu\Upsilon}-2\dfrac{\Lambda^2}{\mu^2}+\dfrac{\Lambda^2}{\mu^2}\right)^{\frac{1}{2}}= \frac{\Lambda}{2\mu}\left(\dfrac{2\mu}{\Upsilon}-1\right)^{\frac{1}{2}}~~a.s.
\end{equation}\vspace*{5pt}
By a similar argument, we can also assert that \vspace*{5pt}
\begin{equation}\label{8}
\lim_{t\to\infty}\dfrac{1}{t}\mathlarger{\int_{0}^t}\left(\widehat{\Psi}(s)-\dfrac{\hat{\Lambda}}{\hat{\mu}}\right)^+\textup{d}s \leqslant \frac{1}{2}\left(\dfrac{2\hat{\Lambda}^2}{\hat{\mu}\widehat{\Upsilon}}-2\dfrac{\hat{\Lambda}^2}{\hat{\mu}^2}+\dfrac{\hat{\Lambda}^2}{\hat{\mu}^2}\right)^{\frac{1}{2}}= \frac{\hat{\Lambda}}{2\hat{\mu}}\left(\dfrac{2\hat{\mu}}{\widehat{\Upsilon}}-1\right)^{\frac{1}{2}}~~a.s.
\end{equation}
It is fairly easy to see that $\mathbf{M}_1(t)$ is a local martingale with finite quadratic variation, and from the hypothesis  $\mathbf{(A_5)}$  we can affirm that  $\mathbf{M}_2(t)$ will be also so. Therefore, we conclude by the strong law of large numbers for local martingales that  
\begin{equation}\label{9}
\lim_{t\to \infty}\dfrac{\mathbf{M}_1(t)}{t}~~\text{and}~~\lim_{t\to \infty}\dfrac{\mathbf{M}_2(t)}{t}~~~~\text{a.s.}
\end{equation}
Taking the superior limit on both sides of \eqref{6} and combining the resulting inequality  with \eqref{7}, \eqref{8} and \eqref{9}  lead us to 
\begin{align*}
\limsup_{t\to\infty}\dfrac{W\big(I(t),\widehat{I}(t)\big)}{t}&\leqslant \mathfrak{D}-\mathfrak{C}-\mathfrak{B}+\dfrac{\lambda_1 b \beta}{\lambda_2}\frac{\Lambda}{2\mu}\left(\dfrac{2\mu}{\Upsilon}-1\right)^{\frac{1}{2}}+\dfrac{\lambda_2 b \hat{\beta}}{\lambda_1}\frac{\hat{\Lambda}}{2\hat{\mu}}\left(\dfrac{2\hat{\mu}}{\widehat{\Upsilon}}-1\right)^{\frac{1}{2}}\\
&=\mathfrak{D}-\mathfrak{C}-\mathfrak{B}+\dfrac{\hat{\mu}\sqrt{\mathcal{R}_0}}{2}\left(\dfrac{2\mu}{\Upsilon}-1\right)^{\frac{1}{2}}+\dfrac{(\mu+\rho)\sqrt{\mathcal{R}_0}}{2}\left(\dfrac{2\hat{\mu}}{\widehat{\Upsilon}}-1\right)^{\frac{1}{2}}=\mathlarger{\mathlarger{\kappa}},
\end{align*}
which is exactly the desired conclusion. In addition, it goes without saying that if $\mathlarger{\mathlarger{\mathlarger{\kappa}}}<0$ then the disease will die out exponentially almost surely. Thus, the  theorem is proved. \vspace*{10pt}
\end{proof}
\begin{rema}
Compared to several existing works (see for instance \cite{liu2018stationary, liu2019dynamics,  liu2020siri, zhou2020stationary, liu2020dynamical, han2020stationary}), the statement of the last theorem is indeed stronger because it offers a sharper threshold that weakens the disease extinction condition. The precision of our threshold $\mathlarger{\mathlarger{\kappa}}$ comes back essentially to inequality \eqref{5.5} in the previous proof where we used the ramp function $(x\mapsto x^+)$ instead of the absolute value always adopted in the literature to our best knowledge.
\end{rema}
\section{Persistence in the mean of the dengue disease}\label{sec3}
After having studied the extinction of the dengue fever, we turn now to explore its persistence in the mean, but before doing so, let us first recall the definition of this notion. 
\begin{defi}[Persistence in the mean \cite{kiouach2021advanced}]
The infectious individuals $I(t)$ and $\widehat{I}(t)$ of the system \eqref{sto1}, are said to \\[5pt] be persistent in the mean if $~~\displaystyle{\liminf\limits_{t\to\infty}\dfrac{1}{t}\int_0^t \left(I(s)+\widehat{I}(s)\right)\textup{d}s>0}$ almost surely. 
\end{defi}
For the sake of greater clarity and readability, we use from now on these notations: \vspace*{15pt}
\begin{itemize}
 \item [$\bullet$]$M_1:=\left(\mu\hspace{-1pt}+\hspace{-1pt}\dfrac{\sigma_1^2}{2}+\displaystyle{\int_{\mathbb{U}}\Big(\xi_1(u)-\ln\big(1+\xi_1(u)\big)\Big)\nu(\textup{d}u)}\right).$\vspace*{3pt}
 \item [$\bullet$]$\displaystyle{M_2:=\left(\left(\mu+\rho\right)+\dfrac{\sigma_2^2}{2}+\int_{\mathbb{U}}\Big(\xi_2(u)-\ln\big(1+\xi_2(u)\big)\Big)\nu(\textup{d}u)\right)}.$\vspace*{3pt}
 \item [$\bullet$]$\displaystyle{M_3:=\left(\hat{\mu}+\dfrac{\sigma_3^2}{2}+\int_{\mathbb{U}}\Big(\xi_3(u)-\ln\big(1+\xi_3(u)\big)\Big)\nu(\textup{d}u)\right)}.$\vspace*{3pt}
 \item [$\bullet$]$\displaystyle{M_4:=\left(\hat{\mu}+\dfrac{\sigma_4^2}{2}+\int_{\mathbb{U}}\Big(\xi_4(u)-\ln\big(1+\xi_4(u)\big)\Big)\nu(\textup{d}u)\right)}.$\vspace*{3pt}
 \item [$\bullet$]$\widetilde{\mathcal{R}}_0:=\displaystyle{\dfrac{b^2\beta\hat{\beta}\Lambda\hat{\Lambda}}{M_1M_2M_3M_4}}.$
\end{itemize} \vspace*{15pt}
\begin{theo}\label{th2}
 Suppose that the assumptions $\mathbf{(A_1)}-\mathbf{(A_5)}$ are verified, and let $\big(\hspace{-0.5pt}S(t),\hspace{-0.5pt}I(t),\hspace{-0.5pt}\widehat{S}(t),\hspace{-0.5pt}\widehat{I}(t)\big)$\hspace{-2pt} be the\hspace{-1pt} solution\hspace{-0.5pt} of system \eqref{sto1} that\hspace{-1pt} starts\hspace{-1pt} from\hspace{-1pt} an \hspace{-1pt} initial value $\hspace{-2pt}$ $\big(\hspace{-0.5pt}S_0,\hspace{-0.5pt}I_0,\hspace{-0.5pt}\widehat{S}_0,\hspace{-0.5pt}\widehat{I}_0\big)$ belonging to the positive orthant  $\mathbb{R}_+^4$.  If the inequality $\widetilde{\mathcal{R}}_0>1$ is true, then the dengue disease presented by equation \eqref{sto1} will persist in the mean almost surely.  
\end{theo}
\begin{proof}
Let us consider the function  $\overline{W}$ defined by \vspace*{5pt}
$$\begin{array}{crcl}
\overline{W}:&\mathbb{R}^4_{+}&\longrightarrow &\mathbb{R}\\ 
 &x&\longmapsto &-\sum\limits_{i=1}^4 \theta_i\ln\left(x_i\right),
\end{array}$$ 
\vspace*{5pt}
where $\theta_2=1$ and $\theta_1, \theta_3, \theta_4$ are three positive constants to be chosen suitably later. From It\^{o}'s formula, we have and for all $t\geqslant0$ \vspace*{5pt}
\begin{align*}
\text{d}\overline{W}\big(S(t),I(t),\widehat{S}(t),\widehat{I}(t)\big)&=\mathcal{L}\overline{W}\big(S(t),I(t),\widehat{S}(t),\widehat{I}(t)\big)~\text{d}t-\theta_1\sigma_1\textup{d}\mathcal{B}_1(t)-\sigma_2\textup{d}\mathcal{B}_2(t)-\theta_3\sigma_3\textup{d}\mathcal{B}_3(t)-\theta_4\sigma_4\textup{d}\mathcal{B}_4(t)\\
&\quad-\theta_1\hspace{-2pt}\int_{\mathbb{U}}\ln\big(1+\xi_1(u)\big)\widetilde{\mathcal{N}}(\textup{d}t,\textup{d}u)-\hspace{-3pt}\int_{\mathbb{U}}\ln\big(1+\xi_2(u)\big)\widetilde{\mathcal{N}}(\textup{d}t,\textup{d}u)-\theta_3\hspace{-2pt}\int_{\mathbb{U}}\ln\big(1+\xi_3(u)\big)\widetilde{\mathcal{N}}(\textup{d}t,\textup{d}u)\\
&\quad-\theta_4\int_{\mathbb{U}}\ln\big(1+\xi_4(u)\big)\widetilde{\mathcal{N}}(\textup{d}t,\textup{d}u),
\end{align*}\vspace*{5pt}
where $\mathcal{L}\overline{W}\big(S(t),I(t),\widehat{S}(t),\widehat{I}(t)\big)$ is given by \vspace*{5pt}
\begin{align*}
\mathcal{L}\overline{W}\big(S(t),I(t),\widehat{S}(t),\widehat{I}(t)\big)&=-\theta_1\dfrac{\Lambda}{S(t)}+\theta_1b\beta\widehat{I}(t)+\theta_1\mu+\theta_1\dfrac{\sigma_1^2}{2}+\theta_1\int_{\mathbb{U}}\Big(\xi_1(u)-\ln\big(1+\xi_1(u)\big)\Big)\nu(\textup{d}u)\\
&\quad-\dfrac{b\beta S(t)\widehat{I}(t)}{I(t)}+\left(\mu+\rho\right)+\dfrac{\sigma_2^2}{2}+\int_{\mathbb{U}}\Big(\xi_2(u)-\ln\big(1+\xi_2(u)\big)\Big)\nu(\textup{d}u)\\
&\quad-\theta_3\dfrac{\hat{\Lambda}}{\widehat{S}(t)}+\theta_3b\hat{\beta}I(t)+\theta_3\hat{\mu}+\theta_3\dfrac{\sigma_3^2}{2}+\theta_3\int_{\mathbb{U}}\Big(\xi_3(u)-\ln\big(1+\xi_3(u)\big)\Big)\nu(\textup{d}u)\\
&\quad -\dfrac{\theta_4 b \hat{\beta} \widehat{S}(t)I(t)}{\widehat{I}(t)}+\theta_4\hat{\mu}+\theta_4\dfrac{\sigma_4^2}{2}+\theta_4\int_{\mathbb{U}}\Big(\xi_4(u)-\ln\big(1+\xi_4(u)\big)\Big)\nu(\textup{d}u).
\end{align*}\vspace*{5pt}
After some simplifications, we get \vspace*{5pt}
\begin{align*}
\mathcal{L}\overline{W}\big(S(t),I(t),\widehat{S}(t),\widehat{I}(t)\big)&=-\theta_1\dfrac{\Lambda}{S(t)}-\dfrac{b\beta S(t)\widehat{I}(t)}{I(t)}-\theta_3\dfrac{\hat{\Lambda}}{\widehat{S}(t)}-\dfrac{\theta_4 b \hat{\beta} \widehat{S}(t)I(t)}{\widehat{I}(t)}+\theta_1b\beta\widehat{I}(t)+\theta_3b\hat{\beta}I(t)\\
&\quad+\theta_1\left(\mu+\dfrac{\sigma_1^2}{2}+\int_{\mathbb{U}}\Big(\xi_1(u)-\ln\big(1+\xi_1(u)\big)\Big)\nu(\textup{d}u)\right)\\
&\quad+\left(\left(\mu+\rho\right)+\dfrac{\sigma_2^2}{2}+\int_{\mathbb{U}}\Big(\xi_2(u)-\ln\big(1+\xi_2(u)\big)\Big)\nu(\textup{d}u)\right)\\
&\quad+\theta_3\left(\hat{\mu}+\dfrac{\sigma_3^2}{2}+\int_{\mathbb{U}}\Big(\xi_3(u)-\ln\big(1+\xi_3(u)\big)\Big)\nu(\textup{d}u)\right)\\
&\quad+\theta_4\left(\hat{\mu}+\dfrac{\sigma_4^2}{2}+\int_{\mathbb{U}}\Big(\xi_4(u)-\ln\big(1+\xi_4(u)\big)\Big)\nu(\textup{d}u)\right).
\end{align*} 
From the relation between geometric and arithmetic  means (the first is less than or equal to the second), it follows that \vspace*{8pt}
\begin{align*}
\mathcal{L}\overline{W}\big(S(t),I(t),\widehat{S}(t),\widehat{I}(t)\big)&\leqslant -4\sqrt[4]{\theta_1\theta_3\theta_4b^2\beta\hat{\beta}\Lambda\hat{\Lambda}}+\theta_1b\beta\widehat{I}(t)+\theta_3b\hat{\beta}I(t)\\
&\quad +\theta_1\underbrace{\left(\mu\hspace{-1pt}+\hspace{-1pt}\dfrac{\sigma_1^2}{2}+\hspace{-5pt}\int_{\mathbb{U}}\Big(\xi_1(u)-\ln\big(1+\xi_1(u)\big)\Big)\nu(\textup{d}u)\right)}_{:=M_1}\\
&\quad+\underbrace{\left(\left(\mu+\rho\right)+\dfrac{\sigma_2^2}{2}+\int_{\mathbb{U}}\Big(\xi_2(u)-\ln\big(1+\xi_2(u)\big)\Big)\nu(\textup{d}u)\right)}_{:=M_2}\\
&\quad+\theta_3\underbrace{\left(\hat{\mu}+\dfrac{\sigma_3^2}{2}+\int_{\mathbb{U}}\Big(\xi_3(u)-\ln\big(1+\xi_3(u)\big)\Big)\nu(\textup{d}u)\right)}_{:=M_3}\\
&\quad+\theta_4\underbrace{\left(\hat{\mu}+\dfrac{\sigma_4^2}{2}+\int_{\mathbb{U}}\Big(\xi_4(u)-\ln\big(1+\xi_4(u)\big)\Big)\nu(\textup{d}u)\right)}_{:=M_4}.\\
\end{align*}
By taking $\theta_1=\dfrac{b^2\beta\hat{\beta}\Lambda\hat{\Lambda}}{M_1^2M_3M_4}$, $\theta_3=\dfrac{b^2\beta\hat{\beta}\Lambda\hat{\Lambda}}{M_1M_3^2M_4}$ and $\theta_4=\dfrac{b^2\beta\hat{\beta}\Lambda\hat{\Lambda}}{M_1M_3M_4^2}$, we obtain:\vspace*{-5pt}
$$\mathcal{L}\overline{W}\big(S(t),I(t),\widehat{S}(t),\widehat{I}(t)\big)\leqslant \theta_1b\beta\widehat{I}(t)+\theta_3b\hat{\beta}I(t)+M_2-\dfrac{b^2\beta\hat{\beta}\Lambda\hat{\Lambda}}{M_1M_3M_4}=b\left(\theta_1\beta\widehat{I}(t)+\theta_3\hat{\beta}I(t)\right)-M_2\Big(\overbrace{\dfrac{b^2\beta\hat{\beta}\Lambda\hat{\Lambda}}{M_1M_2M_3M_4}}^{\widetilde{\mathcal{R}}_0}-1\Big).$$ \vspace*{-5pt}
Therefore\vspace*{8pt}
\begin{align}\label{10}\hspace{-1pt}
\text{d}\overline{W}\big(S(t),I(t),\widehat{S}(t),\widehat{I}(t)\big)\hspace{-1pt}&\leqslant \hspace{-1pt}\left[b\left(\theta_1\beta\widehat{I}(t)+\theta_3\hat{\beta}I(t)\right)\hspace*{-2pt}-\hspace*{-2pt}M_2\Big(\widetilde{\mathcal{R}}_0\hspace{-1pt}-\hspace{-1pt}1\Big)\right]\hspace{-1pt}\text{d}t\hspace*{-2pt}-\hspace*{-2pt}\theta_1\sigma_1\textup{d}\mathcal{B}_1(t)\hspace*{-2pt}-\hspace*{-2pt}\sigma_2\textup{d}\mathcal{B}_2(t)\hspace{-1pt}-\hspace{-1pt}\theta_3\sigma_3\textup{d}\mathcal{B}_3(t)\hspace{-1pt}-\hspace{-1pt}\theta_4\sigma_4\textup{d}\mathcal{B}_4(t)\nonumber\\
&\quad-\theta_1\hspace{-2pt}\int_{\mathbb{U}}\ln\big(1+\xi_1(u)\big)\widetilde{\mathcal{N}}(\textup{d}t,\textup{d}u)-\hspace{-3pt}\int_{\mathbb{U}}\ln\big(1+\xi_2(u)\big)\widetilde{\mathcal{N}}(\textup{d}t,\textup{d}u)-\theta_3\hspace{-2pt}\int_{\mathbb{U}}\ln\big(1+\xi_3(u)\big)\widetilde{\mathcal{N}}(\textup{d}t,\textup{d}u)\nonumber\\
&\quad-\theta_4\int_{\mathbb{U}}\ln\big(1+\xi_4(u)\big)\widetilde{\mathcal{N}}(\textup{d}t,\textup{d}u).
\end{align}\vspace*{8pt}
Integrating \eqref{10} from $0$ to $t$, and then dividing by $t$ on both sides, we get \vspace*{10pt}
\begin{align*}
\dfrac{\overline{W}\big(S(t),I(t),\widehat{S}(t),\widehat{I}(t)\big)}{t}-\dfrac{\overline{W}\big(S(0),I(0),\widehat{S}(0),\widehat{I}(0)\big)}{t}&\leqslant \dfrac{b\theta_1\beta}{t}\int_0^t \widehat{I}(s)~\text{d}s+\dfrac{b\theta_3\hat{\beta}}{t}\int_0^t I(s)~\text{d}s-M_2\Big(\widetilde{\mathcal{R}}_0-1\Big)\\
&\quad-\sum_{i=1}^4\theta_i\sigma_i\dfrac{\mathcal{B}_i(t)}{t}-\sum_{i=1}^4\dfrac{\theta_i}{t}\int_0^t\int_{\mathbb{U}}\ln\big(1+\xi_i(u)\big)\widetilde{\mathcal{N}}(\textup{d}s,\textup{d}u)\\
\end{align*}
So
\begin{align*}
\dfrac{\overline{W}\big(S(t),I(t),\widehat{S}(t),\widehat{I}(t)\big)}{t}-\dfrac{\overline{W}\big(S(0),I(0),\widehat{S}(0),\widehat{I}(0)\big)}{t}&\leqslant \dfrac{b\left(\theta_1\beta\vee \theta_3\hat{\beta}\right)}{t}\int_0^t \left(\widehat{I}(s)+ I(s)\right)\:\text{d}s-M_2\Big(\widetilde{\mathcal{R}}_0-1\Big)\\
&\quad-\sum_{i=1}^4\theta_i\sigma_i\dfrac{\mathcal{B}_i(t)}{t}-\sum_{i=1}^4\dfrac{\theta_i}{t}\int_0^t\int_{\mathbb{U}}\ln\big(1+\xi_i(u)\big)\widetilde{\mathcal{N}}(\textup{d}s,\textup{d}u).
\end{align*}
Hence
\begin{align}\label{11}
\dfrac{1}{t}\int_0^t \left(\widehat{I}(s)+ I(s)\right)\:\text{d}s&\geqslant \dfrac{M_2\Big(\widetilde{\mathcal{R}}_0-1\Big)}{b\left(\theta_1\beta\vee \theta_3\hat{\beta}\right)}+\dfrac{\overline{W}\big(S(t),I(t),\widehat{S}(t),\widehat{I}(t)\big)}{b\left(\theta_1\beta\vee \theta_3\hat{\beta}\right)t}-\dfrac{\overline{W}\big(S(0),I(0),\widehat{S}(0),\widehat{I}(0)\big)}{b\left(\theta_1\beta\vee \theta_3\hat{\beta}\right)t}\nonumber\\
&\quad+\sum_{i=1}^4\theta_i\sigma_i\dfrac{\mathcal{B}_i(t)}{b\left(\theta_1\beta\vee \theta_3\hat{\beta}\right)t}+\sum_{i=1}^4\dfrac{\theta_i}{b\left(\theta_1\beta\vee \theta_3\hat{\beta}\right)t}\int_0^t\int_{\mathbb{U}}\ln\big(1+\xi_i(u)\big)\widetilde{\mathcal{N}}(\textup{d}s,\textup{d}u).
\end{align}
Since $\ln(a)\leqslant a-1\leqslant a$ for any $a>0$, one can conclude that $\overline{W}(x)\geqslant -\sum\limits_{i=1}^4\theta_i x_i$ for all $x\in\mathbb{R}_+^4$. Combining this fact with \eqref{11} leads  to \vspace*{5pt}
\begin{align}\label{11.5}
\dfrac{1}{t}\int_0^t \left(\widehat{I}(s)+ I(s)\right)\:\text{d}s&\geqslant \dfrac{M_2\Big(\widetilde{\mathcal{R}}_0-1\Big)}{b\left(\theta_1\beta\vee \theta_3\hat{\beta}\right)}-\dfrac{1}{t}\times\overbrace{\dfrac{\theta_1S(t)+I(t)+\theta_3\widehat{S}(t)+\theta_1\widehat{I}(t)}{b\left(\theta_1\beta\vee \theta_3\hat{\beta}\right)}}^{:=\mathcal{M}_0(t)}-\dfrac{\overline{W}\big(S(0),I(0),\widehat{S}(0),\widehat{I}(0)\big)}{b\left(\theta_1\beta\vee \theta_3\hat{\beta}\right)t}\nonumber\\
&\quad+\dfrac{1}{t}\times\underbrace{\sum_{i=1}^4\theta_i\sigma_i\dfrac{\mathcal{B}_i(t)}{b\left(\theta_1\beta\vee \theta_3\hat{\beta}\right)}}_{:=\mathcal{M}_1(t)}+\dfrac{1}{t}\times\underbrace{\sum_{i=1}^4\dfrac{\theta_i}{b\left(\theta_1\beta\vee \theta_3\hat{\beta}\right)}\int_0^t\int_{\mathbb{U}}\ln\big(1+\xi_i(u)\big)\widetilde{\mathcal{N}}(\textup{d}s,\textup{d}u)}_{:=\mathcal{M}_2(t)}.
\end{align}\vspace*{5pt}
Needless to say, $\mathcal{M}_1(t)$ is a local martingale with finite quadratic variation, and from the assumption  $\mathbf{(A_5)}$  we can assert that  $\mathcal{M}_2(t)$ is also so. Therefore, we deduce by the strong law of large numbers for local martingales that\begin{equation}\label{12}
\lim_{t\to \infty}\dfrac{\mathcal{M}_1(t)}{t}~~\text{and}~~\lim_{t\to \infty}\dfrac{\mathcal{M}_2(t)}{t}~~~~\text{a.s.}
\end{equation}
On the other hand it is clear by virtue of Lemma \ref{lem} that 
\begin{equation}\label{13}
\lim_{t\to \infty}\dfrac{\mathcal{M}_0(t)}{t}~~~~\text{a.s.}
\end{equation}
Taking the inferior limit on both sides of \eqref{11.5} and combining the resulting inequality  with \eqref{12} and \eqref{13}  yields
\begin{equation}
\liminf_{t\to\infty}\dfrac{1}{t}\int_0^t \left(\widehat{I}(s)+ I(s)\right)\:\text{d}s\geqslant \dfrac{M_2\Big(\widetilde{\mathcal{R}}_0-1\Big)}{b\left(\theta_1\beta\vee \theta_3\hat{\beta}\right)}>0~~a.s.
\end{equation}
So if $\widetilde{\mathcal{R}}_0>1$, then the disease will persist in the mean as claimed, which completes the proof. 
\end{proof}
\section{Numerical simulation examples}\label{sec4} \vspace*{0.2cm}
In this section, and by taking the parameter values from the theoretical data presented in Table \ref{tab1}, we set forth some numerical simulations to belay the various results proved in this paper.  The solution of our Dengue model,  is simulated in our case with the initial condition given by $S(0)=0.2, I(0)=0.1,\widehat{S}(0)=0.3$ and $\widehat{I}(0)=0.4$. In what follows, we consider that the unity of time is one day and the number of individuals is expressed in one million population.
\begin{center}\vspace*{1.2cm}
\begin{tabular}{||c||c||c|c||}
\hline
Parameters &Description  &\multicolumn{2}{|c||}{Numerical values} \\
\hline \hline
$\Lambda$ & Humans recruitment rate & 0.5  &0.85\\
\hline
$b$ & Mosquitoes' biting rate  & 3 & 7\\
\hline
$\beta$ & Transmission rate from mosquitoes to humans  &0.15 & 0.65\\
\hline
$\mu$ & Humans' natural death rate &0.8& 0.8\\
\hline
$\rho_0$ & Dengue induced death rate &0.8 &0.8\\
\hline
$\rho_1$ & Heal rate, whether by treatment or naturally &0.02 &0.25\\
\hline
$\hat{\Lambda}$ & Mosquitoes recruitment rate &0.6 &0.6\\
\hline
$\hat{\beta}$ & Transmission rate from humans to mosquitoes  &0.55 &0.55\\
\hline
$\hat{\mu}$ & Mosquitoes' natural death rate &0.9 &0.88\\
\hline
$\sigma_1$ & Intensity of the Brownian motion $\mathcal{B}_1$ &0.269 &0.269\\
\hline
$\sigma_2$ & Intensity of the Brownian motion $\mathcal{B}_2$ &0.25 &0.25\\
\hline
$\sigma_3$ & Intensity of the Brownian motion $\mathcal{B}_3$ &0.25 &0.245\\
\hline
$\sigma_4$ & Intensity of the Brownian motion $\mathcal{B}_4$ &0.13 &0.14\\
\hline
$\xi_1$ & Intensity of the  L\'{e}vy jumps associated to $S$ &-0.75 &-0.75\\
\hline
$\xi_2$ & Intensity of the  L\'{e}vy jumps associated to $I$ &0.8 &0.78\\
\hline
$\xi_3$ & Intensity of the  L\'{e}vy jumps associated to $\widehat{S}$ &-0.9 &-0.9\\
\hline
$\xi_4$ & Intensity of the  L\'{e}vy jumps associated to $\widehat{I}$ &0.85 &0.85\\
\hline
\multicolumn{2}{c||}{} & Figure \ref{Fig1}  & Figure \ref{Fig2}\\ 
\cline{3-4}
\end{tabular}
\captionof{table}{Definitions and nominal values, per day, of the system  parameters and perturbations intensities used in the simulations.}
\label{tab1}
\end{center}\vspace*{1.2cm}
\subsection{The case of dengue fever stochastic extinction}
In this case, and by selecting the parameter values appearing in the third column of Table \ref{tab1},  and also taking  $\mathbb{U}=\mathbb{R}_+$  with $\nu(\mathbb{U})=1$,  we will shed some light on the theoretical results of Section \ref{sec2}. By a few simple calculations, it can be verified that we have the numerical values presented in Table \ref{tab2}. From the latter, we observe easily  that the condition $\mathlarger{\mathlarger{\mathlarger{\kappa}}}<0$, and hypothesis $\mathbf{(A_1)}$, $\mathbf{(A_2)}$, $\mathbf{(A_3)}$ and $\mathbf{(A_4)}$ hold. So, and by the virtue of Theorem \ref{th} the dengue epidemic dies out exponentially almost surely, which is exactly depicted in Figure \ref{Fig1}.\\ \vspace*{10pt}
\begin{flushleft}
\begin{tabular}{l||Sc||Sc||Sc||}
\cline{2-4}
 & Quantity & Expression & Value \\
\cline{2-4}
Assumption $\mathbf{(A_1)}$ $\Bigg\lbrace$ & $\mathfrak{M}_1$ & $\displaystyle{\max_{1\leqslant i \leqslant 4}\left(\int_{\mathbb{U}}\xi_i^2(u)\nu(\text{d}u)\right)}$ & 0.81$<\infty$\\
\cline{2-4}
Assumption $\mathbf{(A_2)}$ $\Bigg\lbrace$ & $\mathfrak{M}_2$ & $\displaystyle{\max_{1\leqslant i \leqslant 4}\left(\int_{\mathbb{U}}\Big(\xi_i(u)-\ln\left(1+\xi_i(u)\right)\Big)\nu(\text{d}u)\right)}$ & 1.4025$<\infty$\\
\cline{2-4}
 & $\Sigma$ & $\max\big\lbrace \sigma_i^2 \mid i\in\{1,2,3,4\}\big\rbrace$ & 0.0724\\
\cline{2-4}
 & $\widetilde{\xi}$ & $\max\big\lbrace \xi_i(u) \mid i\in\{1,2,3,4\} \big\rbrace$ & 0.85\\
\cline{2-4}
 & $\utilde{\xi}$ & $\min\big\lbrace \xi_i(u) \mid i\in\{1,2,3,4\} \big\rbrace$  & -0.9\\
\cline{2-4}
 & $\widetilde{\mathlarger{\mathlarger{\theta}}}_p$ & $\left(1+\widetilde{\xi}\right)^p-p\times\widetilde{\xi}-1$ & 1.5301\\
\cline{2-4}
 & $\utilde{\mathlarger{\mathlarger{\theta}}}_p$ & $\left(1+\utilde{\xi}\right)^p-p\times\utilde{\xi}-1$ & 1.2531\\
\cline{2-4}
& $\mathlarger{\mathlarger{\theta}}_p$ & $\utilde{\mathlarger{\mathlarger{\theta}}}_p \vee \widetilde{\mathlarger{\mathlarger{\theta}}}_p$ & 1.5301\\
\cline{2-4}
& $\varrho_p$ & $\displaystyle{\int_{\mathbb{U}}} \mathlarger{\theta}_p(u) \nu(\text{d}u)$ & 1.5301\\
\cline{2-4}
Assumption $\mathbf{(A_3)}$ $\Bigg\lbrace$& $\Delta_p$ & $\mu\wedge\hat{\mu}-\dfrac{p-1}{2}\Sigma-\dfrac{\varrho_p}{p}$ & 0.13366$>0$\\
\cline{2-4}
Assumption $\mathbf{(A_4)}$ $\Bigg\lbrace$ & $\mathfrak{M}_3$ & $\displaystyle{\max_{1\leqslant i \leqslant 4}\left(\int_{\mathbb{U}}\Big(\big(1+\xi_i(u)\big)^2-1\Big)^2\nu(\text{d}u)\right)}$ & 5.868$<\infty$\\
\cline{2-4}
& $\displaystyle{\Upsilon}$ & $\displaystyle{2\mu-\sigma_1^2-\int_{\mathbb{U}}\xi_1^2(u)\nu(\textup{d}u)}$ & 0.9651\\
\cline{2-4}
& $\displaystyle{\widehat{\Upsilon}}$ & $\displaystyle{2\hat{\mu}-\sigma_3^2-\int_{\mathbb{U}}\xi_3^2(u)\nu(\textup{d}u)}$ & 0.9275\\
\cline{2-4}
& $\underline{\mathfrak{b}}$ & $\big(\xi_2\wedge\xi_4-\ln\left(1+\xi_2\wedge\xi_4\right)\big)\times \mathds{1}_{\left\{\xi_2(u)\wedge\xi_4(u)>0\right\}}$ & 0.2122 \\
\cline{2-4}
& $\overline{\mathfrak{b}}$ & $\big(\xi_2\vee\xi_4-\ln\left(1+\xi_2\vee\xi_4\right)\big)\times \mathds{1}_{\left\{\xi_2(u)\vee\xi_4(u)\leqslant 0\right\}}$ &0\\
\cline{2-4}
& $\mathfrak{B}$ & $\displaystyle{\int_{\mathbb{U}}\Big(\overline{\mathfrak{b}}(u)+\underline{\mathfrak{b}}(u)\Big)~\nu(\textup{d}u)}$& 0.2122\\
\cline{2-4}
&$\mathfrak{C}$ &  $\dfrac{(\sigma_2\sigma_4)^2}{2(\sigma_2^2+\sigma_4^2)}$ & 0.026\\
\cline{2-4}
& $\mathfrak{D}$&$\left[(\mu+\rho)\vee\hat{\mu}\right]\times \left(\sqrt{\mathcal{R}_0}-1\right)^+-\left[(\mu+\rho)\wedge\hat{\mu}\right]\times \left( 1-\sqrt{\mathcal{R}_0}\right)^+$& -0.4854\\
\cline{2-4}
& $\mathcal{R}_0$&$\dfrac{b^2\beta\Lambda\hat{\beta}\hat{\Lambda}}{\mu\left(\mu+\rho_0+\rho_1\right)\hat{\mu}^2}$& 0.2122\\
\cline{2-4}
& $\mathlarger{\mathlarger{\mathlarger{\kappa}}}$ & $\mathfrak{D}-\mathfrak{C}-\mathfrak{B}+\dfrac{\hat{\mu}\sqrt{\mathcal{R}_0}}{2}\left(\dfrac{2\mu}{\Upsilon}-1\right)^{\frac{1}{2}}+\dfrac{(\mu+\rho)\sqrt{\mathcal{R}_0}}{2}\left(\dfrac{2\hat{\mu}}{\widehat{\Upsilon}}-1\right)^{\frac{1}{2}}$&-0.2044$<0$\\
\cline{2-4}
\end{tabular}
\captionof{table}{Some quantities and their corresponding numerical values in the case of $p=2.5>2$.}
\label{tab2} \vspace*{10pt}
\end{flushleft}

\subsection{The case of dengue fever persistence in the mean}
In this subsection, we will turn to the dengue's persistence in the mean case. By keeping in mind the numerical values appearing in the last column of Table \ref{tab1}, we can easily draw up the following list:
\begin{flushleft}
\begin{tabular}{l||Sc||Sc||Sc||}
\cline{2-4}
 & Quantity & Expression & Value \\
\cline{2-4}
Assumption $\mathbf{(A_1)}$ $\Bigg\lbrace$ & $\mathfrak{M}_1$ & $\displaystyle{\max_{1\leqslant i \leqslant 4}\left(\int_{\mathbb{U}}\xi_i^2(u)\nu(\text{d}u)\right)}$ & 0.81$<\infty$\\
\cline{2-4}
Assumption $\mathbf{(A_2)}$ $\Bigg\lbrace$ & $\mathfrak{M}_2$ & $\displaystyle{\max_{1\leqslant i \leqslant 4}\left(\int_{\mathbb{U}}\xi_i(u)-\ln\left(1+\xi_i(u)\right)\nu(\text{d}u)\right)}$ & 1.4025$<\infty$\\
\cline{2-4}
& $\Sigma$ & $\max\big\lbrace \sigma_i^2 \mid i\in\{1,2,3,4\}\big\rbrace$ & 0.0724\\
\cline{2-4}
 & $\widetilde{\xi}$ & $\max\big\lbrace \xi_i(u) \mid i\in\{1,2,3,4\} \big\rbrace$ & 0.85\\
\cline{2-4}
 & $\utilde{\xi}$ & $\min\big\lbrace \xi_i(u) \mid i\in\{1,2,3,4\} \big\rbrace$  &-0.9\\
\cline{2-4}
 & $\widetilde{\mathlarger{\mathlarger{\theta}}}_p$ & $\left(1+\widetilde{\xi}\right)^p-p\times\widetilde{\xi}-1$ & 1.5301\\
\cline{2-4}
 & $\utilde{\mathlarger{\mathlarger{\theta}}}_p$ & $\left(1+\utilde{\xi}\right)^p-p\times\utilde{\xi}-1$ & 1.2531\\
\cline{2-4}
& $\mathlarger{\mathlarger{\theta}}_p$ & $\utilde{\mathlarger{\mathlarger{\theta}}}_p \vee \widetilde{\mathlarger{\mathlarger{\theta}}}_p$ & 1.5301\\
\cline{2-4}
& $\varrho_p$ & $\displaystyle{\int_{\mathbb{U}}} \mathlarger{\theta}_p(u) \nu(\text{d}u)$ & 1.5301\\
\cline{2-4}
Assumption $\mathbf{(A_3)}$ $\Bigg\lbrace$ & $\Delta_p$ & $\mu\wedge\hat{\mu}-\dfrac{p-1}{2}\Sigma-\dfrac{\varrho_p}{p}$ & 1.13366$>0$\\
\cline{2-4}
Assumption $\mathbf{(A_4)}$ $\Bigg\lbrace$ & $\mathfrak{M}_3$ & $\displaystyle{\max_{1\leqslant i \leqslant 4}\left(\int_{\mathbb{U}}\Big(\big(1+\xi_i(u)\big)^2-1\Big)^2\nu(\text{d}u)\right)}$ & 5.868$<\infty$\\
\cline{2-4}
Assumption $\mathbf{(A_5)}$ $\Bigg\lbrace$ & $\mathfrak{M}_4$ & $\displaystyle{\max_{1\leqslant i \leqslant 4}\left(\int_{\mathbb{U}}\Big(\ln\big(1+\xi_i(u)\big)\Big)^2\nu(\textup{d}u)\right)}$ & 0.378\\
\cline{2-4}
& $M_1$ & $\mu\hspace{-1pt}+\hspace{-1pt}\dfrac{\sigma_1^2}{2}+\displaystyle{\int_{\mathbb{U}}\Big(\xi_1(u)-\ln\big(1+\xi_1(u)\big)\Big)\nu(\textup{d}u)}$ & 1.4725\\
\cline{2-4}
& $M_2$ & $\left(\mu+\rho\right)+\dfrac{\sigma_2^2}{2}+\displaystyle{\int_{\mathbb{U}}\Big(\xi_2(u)-\ln\big(1+\xi_2(u)\big)\Big)\nu(\textup{d}u)}$ &2.0935\\
\cline{2-4}
& $M_3$ & $\hat{\mu}+\dfrac{\sigma_3^2}{2}+\displaystyle{\int_{\mathbb{U}}\Big(\xi_3(u)-\ln\big(1+\xi_3(u)\big)\Big)\nu(\textup{d}u)}$ & 2.3338\\
\cline{2-4}
& $M_4$ & $\hat{\mu}+\dfrac{\sigma_4^2}{2}+\displaystyle{\int_{\mathbb{U}}\Big(\xi_4(u)-\ln\big(1+\xi_4(u)\big)\Big)\nu(\textup{d}u)}$ & 1.1433\\
\cline{2-4}
& $\widetilde{\mathcal{R}}_0$&$\displaystyle{\dfrac{b^2\beta\hat{\beta}\Lambda\hat{\Lambda}}{M_1M_2M_3M_4}}$&1.0862$>1$\\
\cline{2-4}
\end{tabular}
\captionof{table}{The corresponding values of some quantities  in the case of $p=2.5>2$.}
\label{tab3}\vspace*{10pt}
\end{flushleft}
On the basis of this last table's numerical values, the assumptions $\mathbf{(A_1)}-\mathbf{(A_5)}$ hold, and the quantity $\widetilde{\mathcal{R}}_0$ outnumbers one. So, and by Theorem \ref{th2}, the dengue fever is persistent in the mean, which agrees well with the pictorial curves presented in Figure \ref{Fig2}.
    
\begin{rema}
Plainly, and in all the simulation examples mentioned above, assumptions $\mathbf{(A_1)}$ and $\mathbf{(A_2)}$ are satisfied. So, and since $\big(S(0),I(0),\widehat{S}(0),\widehat{I}(0)\big)\in\mathbb{R}_+^4$, Theorem \ref{th0} guarantees the positivity of the solution for all  future time $t>0$. This last fact is illustrated and clearly corroborated by the different curves depicted in figures \ref{Fig1} and \ref{Fig2}.
\end{rema}
\begin{rema}
If we kept the same approach as \cite{liu2018stationary, liu2019dynamics,liu2020siri,zhou2020stationary,liu2020dynamical,han2020stationary} in our treatment, the threshold that we will obtain in Theorem \ref{th} will be of the following form $K:=\mathfrak{D}-\mathfrak{C}-\mathfrak{B}+\hat{\mu}\times\sqrt{\mathcal{R}_0\left(\dfrac{2\mu}{\Upsilon}-1\right)}+(\mu+\rho)\times\sqrt{\mathcal{R}_0\left(\dfrac{2\hat{\mu}}{\widehat{\Upsilon}}-1\right)}.$  
The latter, and in the case of the numerical values taken in Figure \ref{Fig1}, is positive, therefore, it will be unable to guarantee the disease's extinction. On the other hand, the new threshold $\mathlarger{\mathlarger{\mathlarger{\kappa}}}$ that we have proposed is well capable of doing it, and this reflects clearly its sharpness.
\end{rema}
\begin{figure}[H]
\centering
\subfigure{
    \includegraphics[width=.365\linewidth]{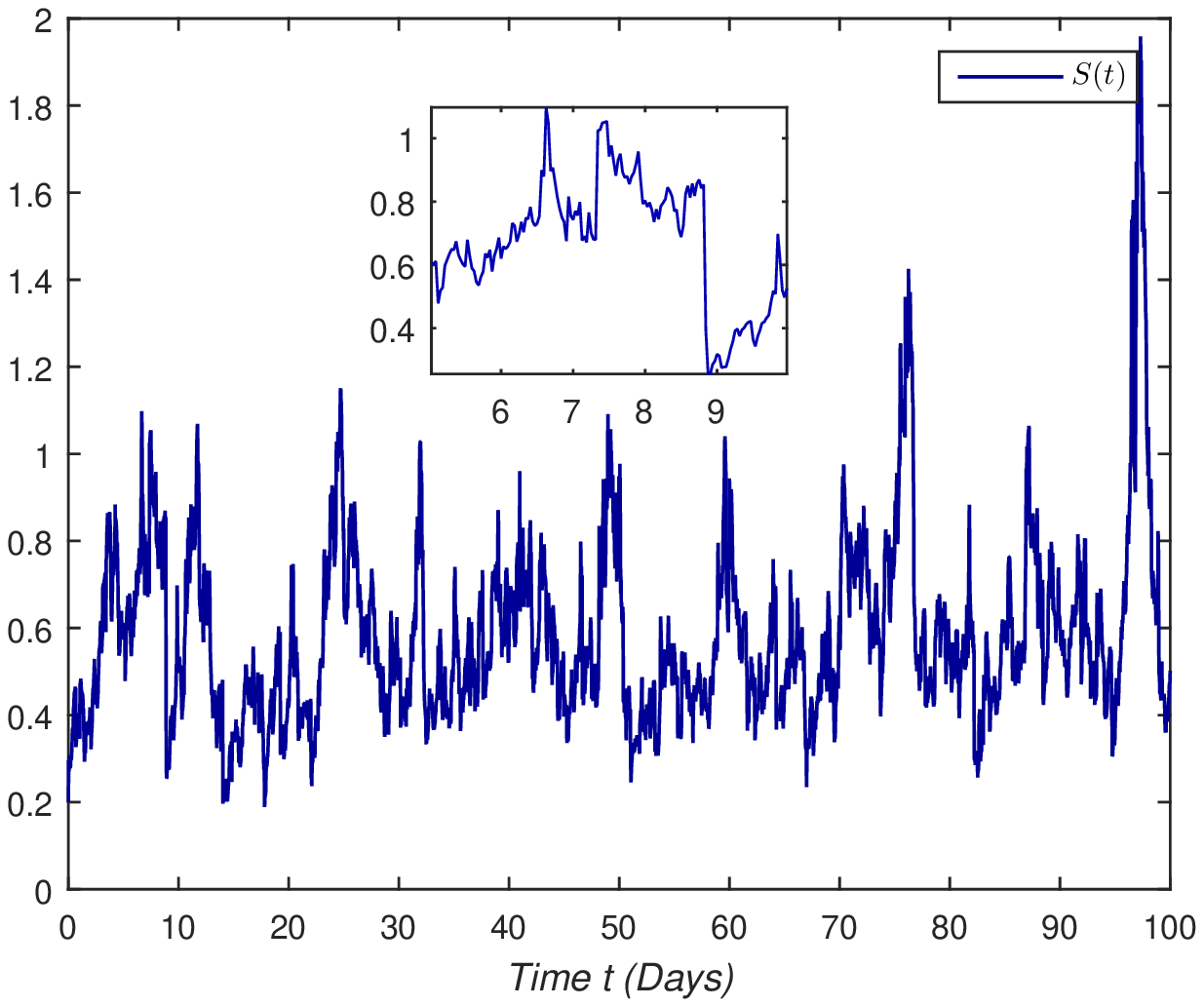}
  }%
\subfigure{
    \includegraphics[width=.365\linewidth]{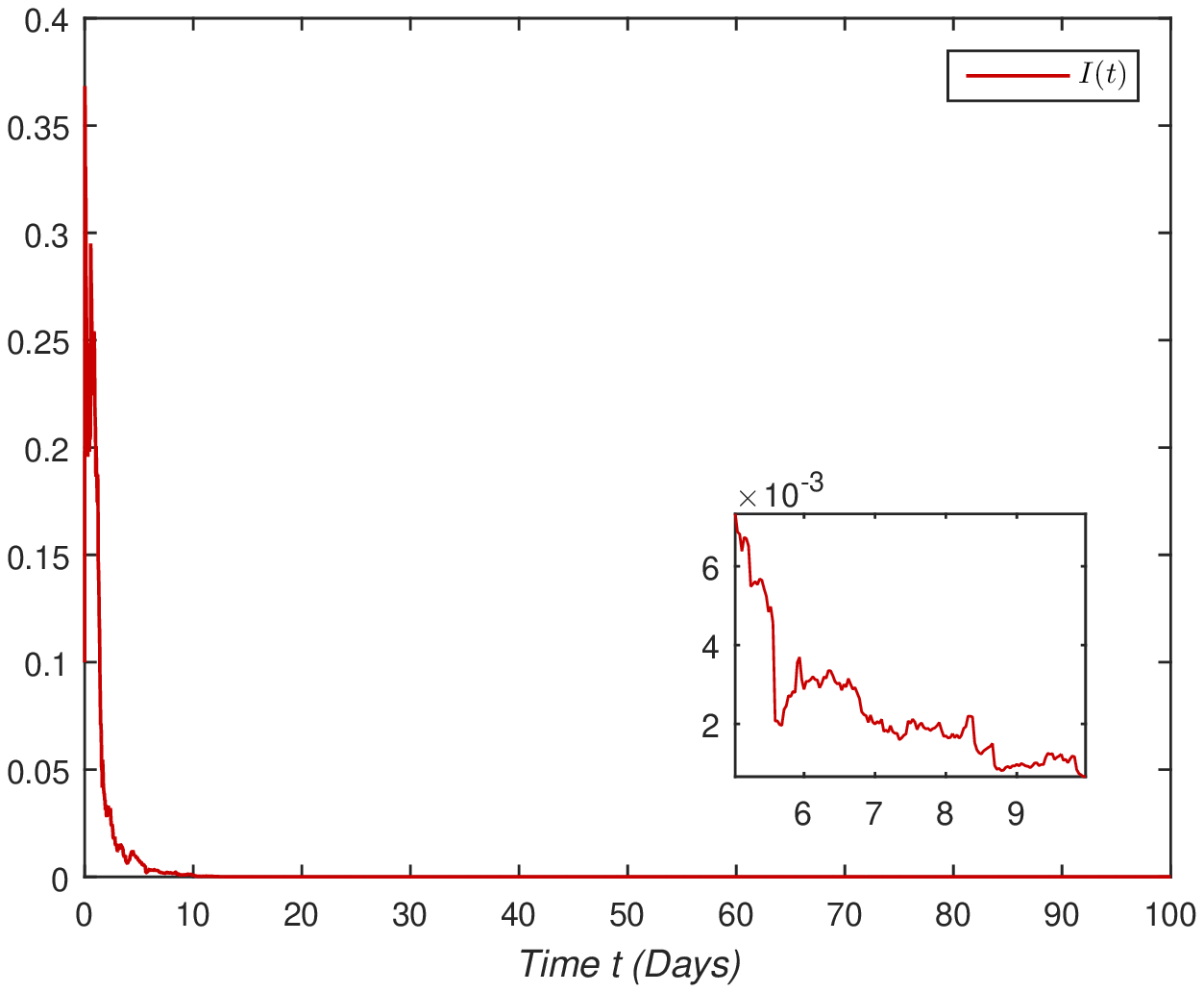}
  }\\[-2pt]
  \subfigure{
    \includegraphics[width=.365\linewidth]{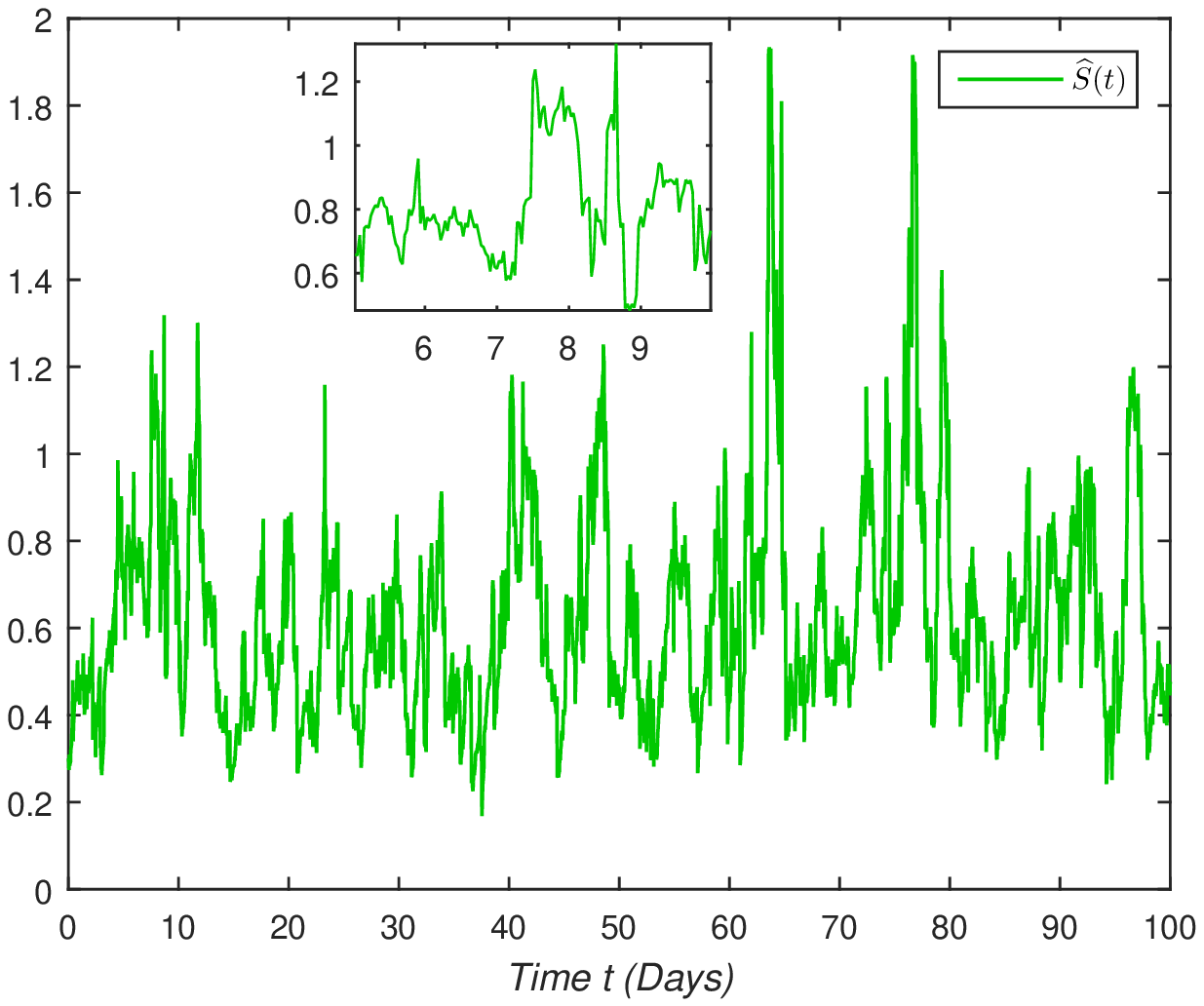}
  }%
\subfigure{
    \includegraphics[width=.365\linewidth]{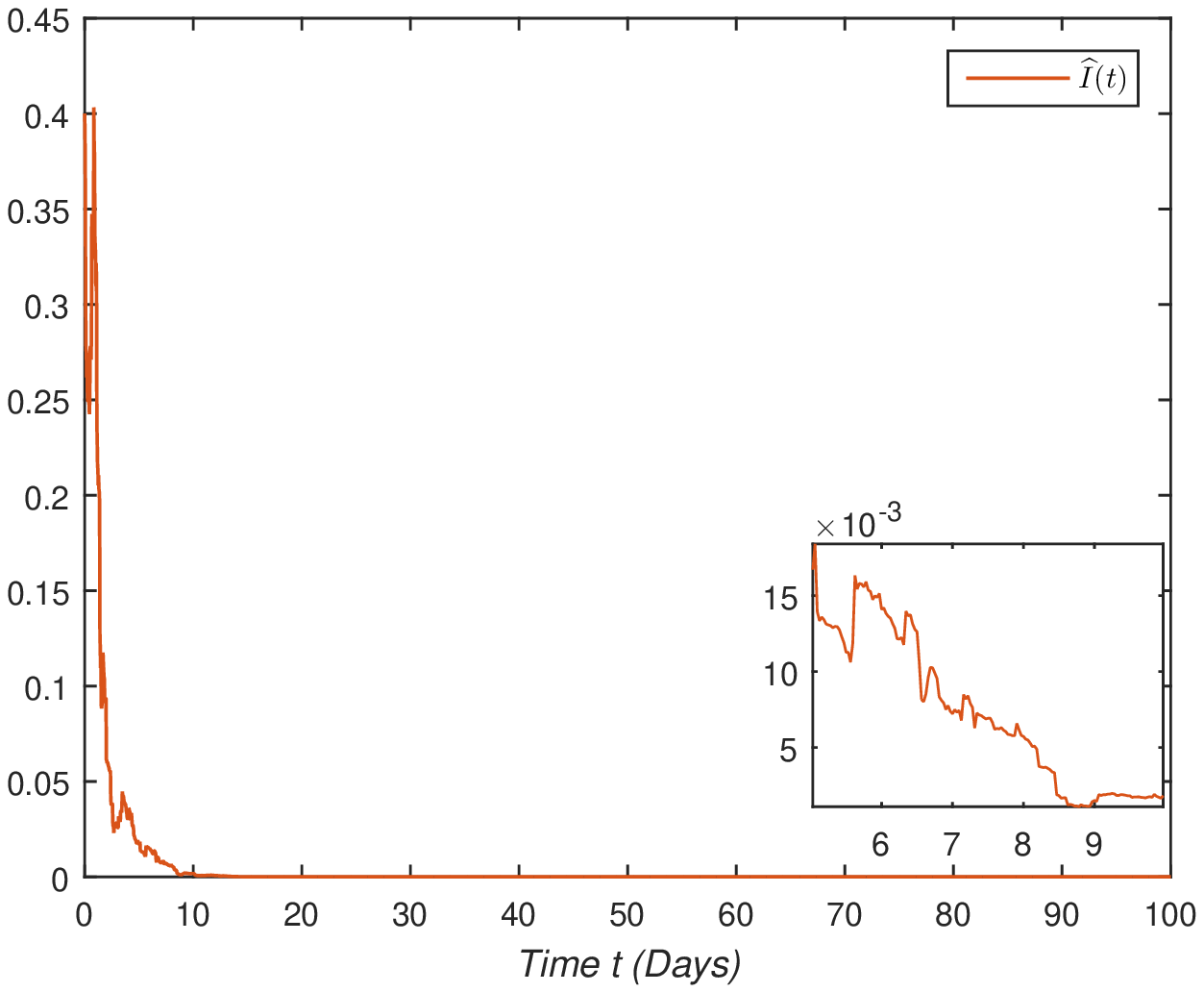}
    }
  \vspace*{-6pt}
 \caption{Solutions' paths of the dengue fever model \eqref{sto1} when the numerical values are taken as shown in the third column of Table \ref{tab1} ($\mathlarger{\mathlarger{\mathlarger{\kappa}}}=-0.2122<0$).}\label{Fig1}
\end{figure}
\begin{figure}[H]
\centering
\subfigure{
    \includegraphics[width=.365\linewidth]{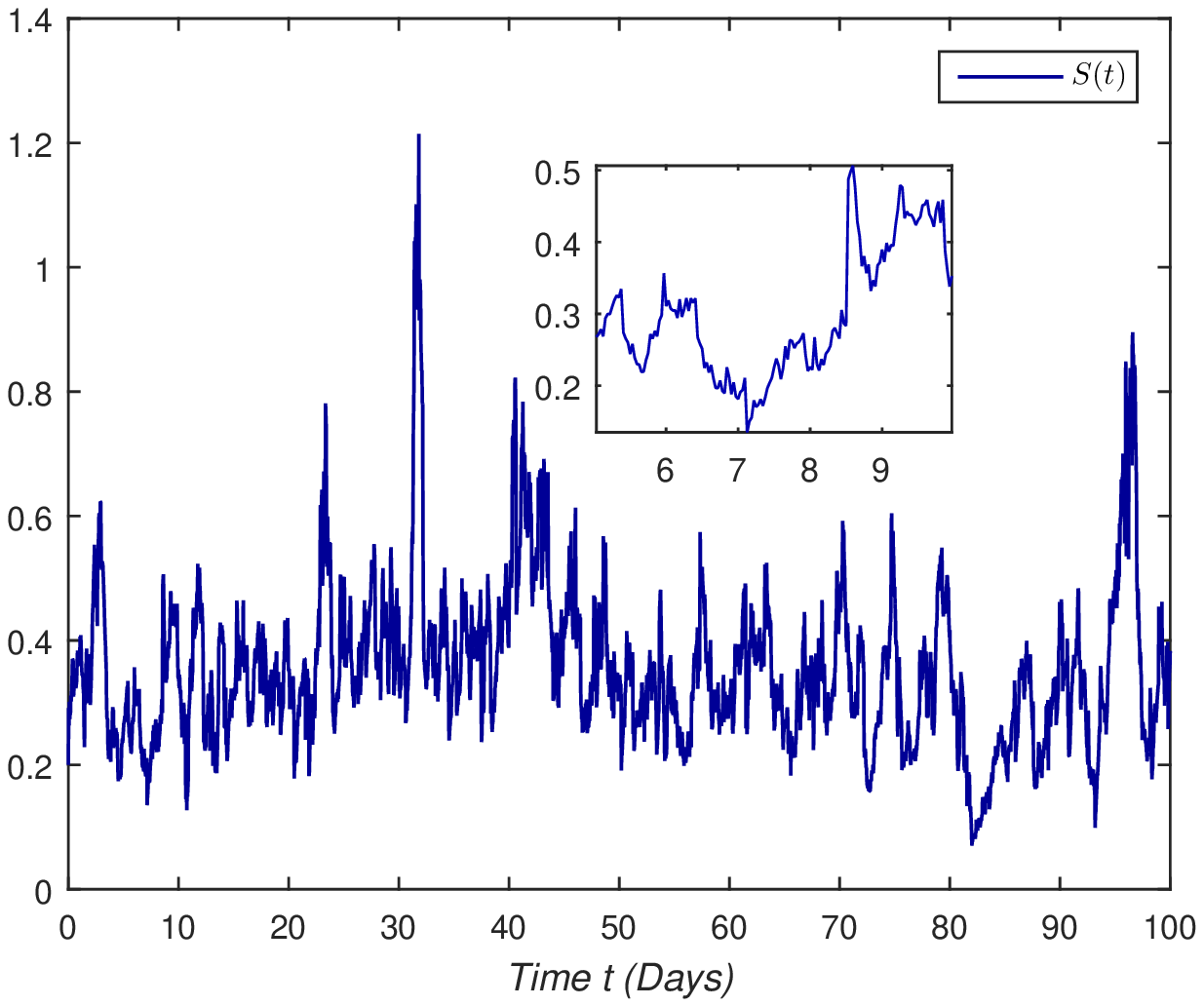}
  }%
\subfigure{
    \includegraphics[width=.365\linewidth]{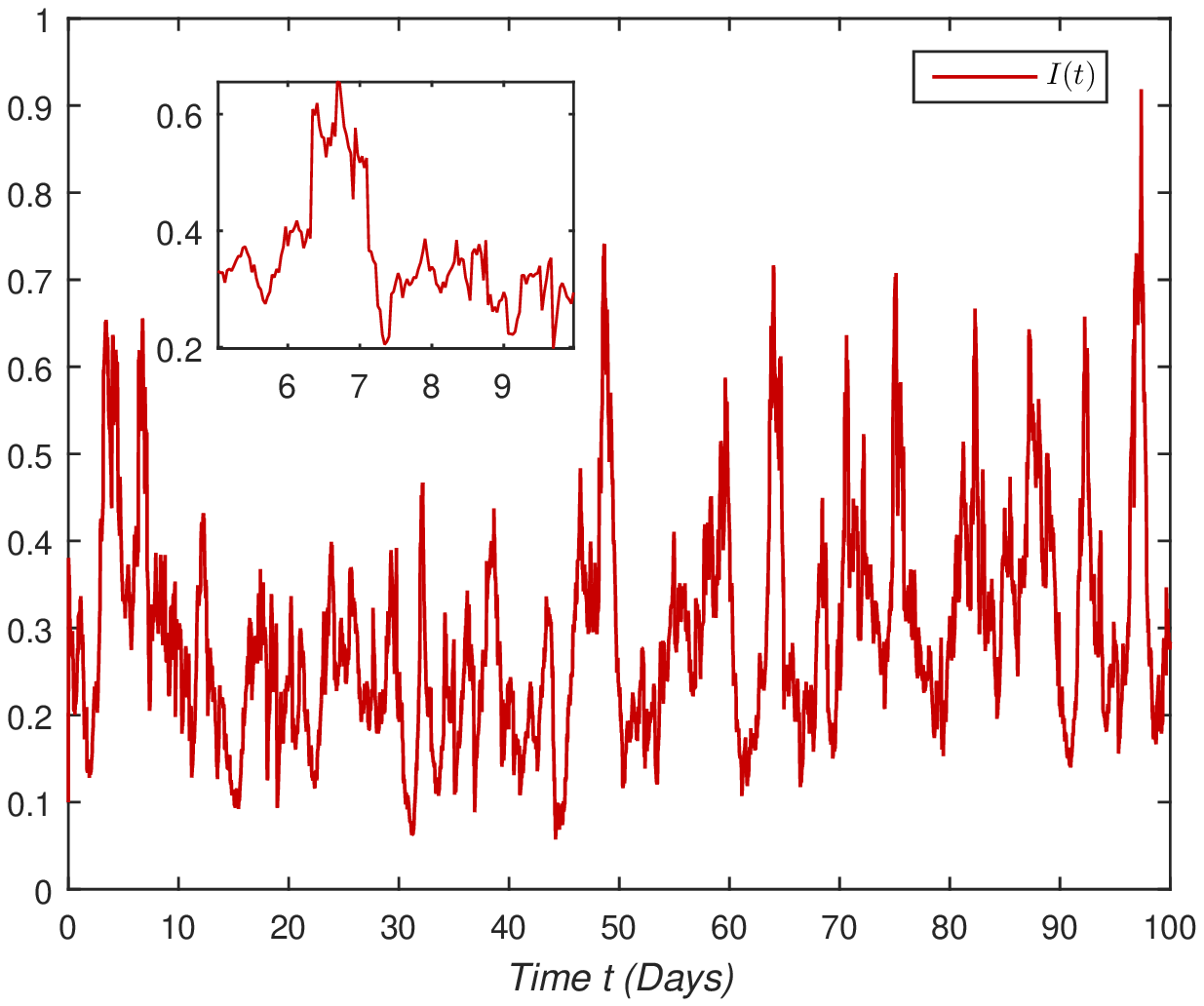}
  }\\[-2pt]
  \subfigure{
    \includegraphics[width=.365\linewidth]{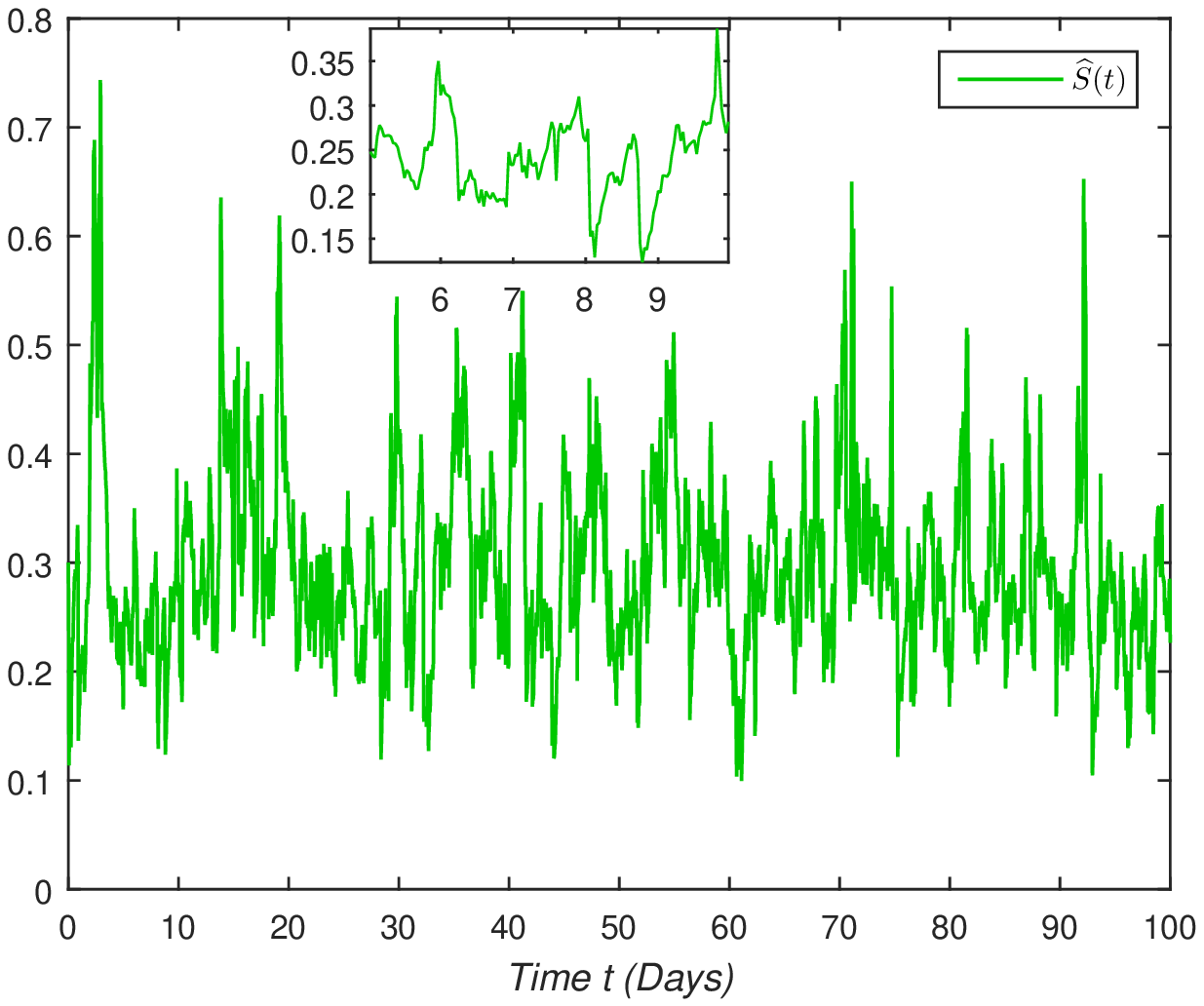}
  }%
\subfigure{
    \includegraphics[width=.365\linewidth]{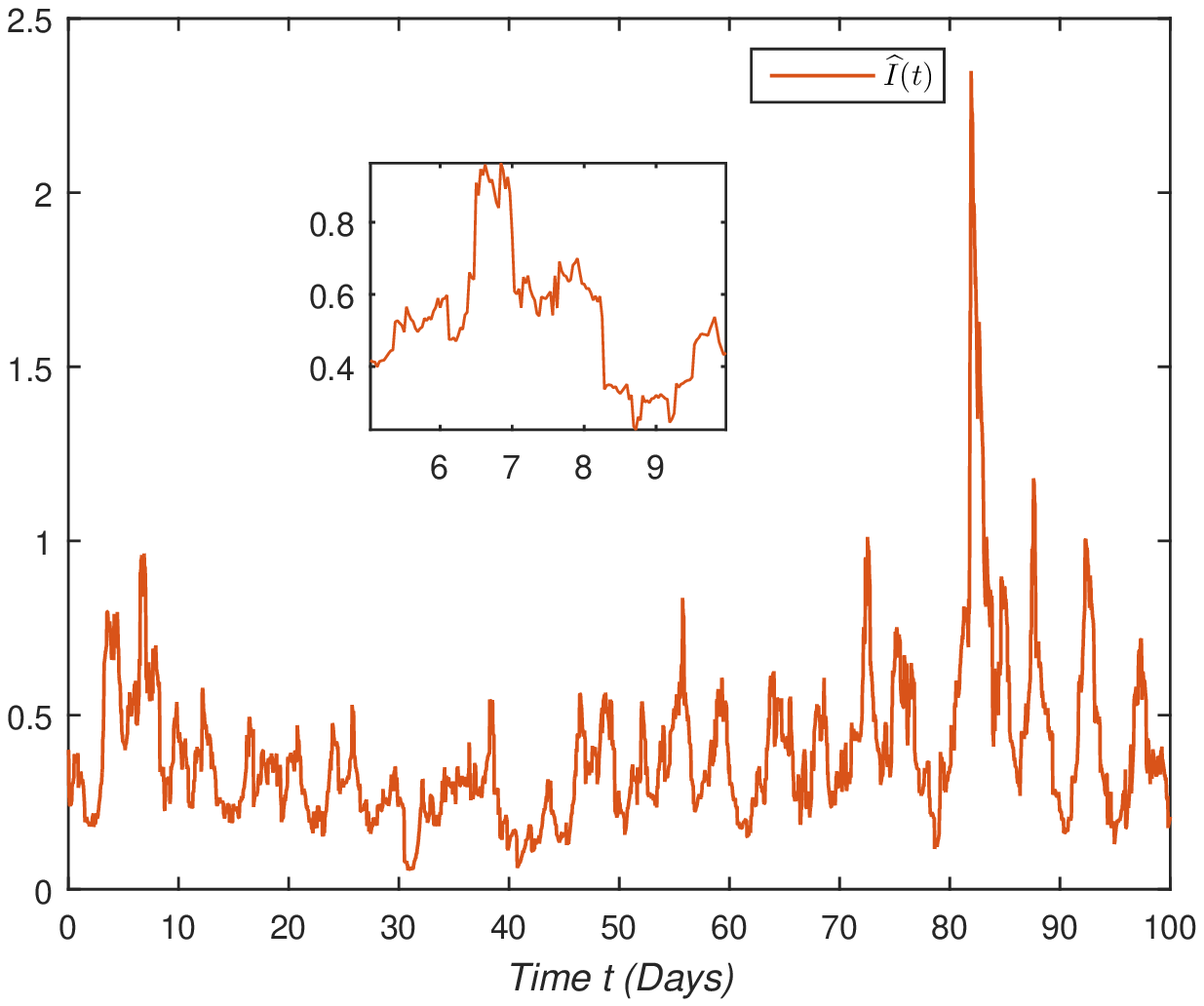}
    }
  \vspace*{-6pt}
 \caption{Solutions' paths of the dengue fever model \eqref{sto1} when the numerical values are chosen as shown in the fourth column of Table \ref{tab1} ($\widetilde{\mathcal{R}}_0=1.0862>1$).}\label{Fig2}
\end{figure}
\section{Conclusion and discussion}\label{sec5}
Dengue fever is an arthropod-borne viral epidemic conveyed to humans through the bite of an infected mosquito with one of the dengue virus's serotypes which all belong to the Flaviviridae viruses family \cite{khan2021dengue}. The dangerousness of this disease lies mainly in its ability to affect almost all age groups, ranging from infants to adults, and in its capacity to re-emerge rapidly in almost every human host population. The absence of an approved treatment or effective vaccine against dengue makes a good comprehension of its prevalence dynamics the only remaining solution to mitigate its severity. In this context, our work presents a mathematical compartmental model that describes the dengue disease dissemination under environmental small disturbances and unexpected massive external perturbations. More explicitly, we have formulated the dengue spread mechanisms by using an SIR-SI stochastic differential equations system that includes both proportional white noises and L\'{e}vy jumps. After having drawn up our model, a rigorous mathematical analysis is performed to get an insight into the dengue fever propagation behavior, especially in the long-term. The principal mathematical and epidemiological findings of our paper are listed as follows:
\begin{itemize}
\item[$\bullet$] We have demonstrated the existence and uniqueness of a positive global-in-time solution to our proposed dengue model.
\item[$\bullet$] We have provided some sufficient conditions for the dengue fever extinction.
\item[$\bullet$] An appropriate hypothetical framework for the persistence in the mean of the dengue disease is also established.
\end{itemize}
Compared to the existing works, the originality of our article resides in the following mathematical techniques and amelioration that we have used to accomplish our analysis:
\begin{itemize}
\item[$\bullet$] Our study adopted an alternative approach to estimate the time averages $\dfrac{\int_0^t\Psi(s)\:\textup{d}s}{t}$, $\dfrac{\int_0^t\Psi^2(s)\:\textup{d}s}{t}$, $\dfrac{\int_0^t\widehat{\Psi}(s)\:\textup{d}s}{t}$  and $\dfrac{\int_0^t\widehat{\Psi}^2(s)\:\textup{d}s}{t}$, without resorting to the explicit stationary distribution's formula which still unknown until \\[5pt] now (see \cite{zhao2018sharp,zhao2019stochastic}).
\item[$\bullet$] In our analysis, we have used the inequality \eqref{5} instead of $\ln(x+1)\leqslant x$ for all $x>-1$, adopted usually in the literature (see for example \cite{cheng2019dynamics}), which allowed us to obtain a sharper threshold for the dengue extinction case.
\item[$\bullet$] Our paper is distinguished from several previous works \cite{liu2018stationary,liu2019dynamics,liu2020siri,zhou2020stationary,liu2020dynamical,han2020stationary} by the use of the ramp function in place of the absolute value in inequality \eqref{5.5}, which remarkably improved the extinction threshold and hence strengthened the statement of Theorem \ref{th}. 
\end{itemize}
Roughly speaking, our theoretical results show that the extinction and persistence conditions depend mainly on the white noise intensities, L\'{e}vy jumps magnitudes, and the system parameters of course. In order to elucidate the theoretical results and exhibit the effect of replacing the absolute value by the ramp function in the inequality \eqref{5.5}, we have presented some numerical simulation examples. In the end, we point out that the obtained results generalize several previous works (for instance, \cite{cai2009global} and \cite{liu2018stationary}), and improve our understanding of the dengue's spreading comportment, which makes this work a good basis for future studies, especially with the continuous reappearance of the dengue fever disease in many regions around the globe. 
\section*{Funding}
This research did not receive any specific grant from funding agencies in the public, commercial, or not-for-profit sectors.
\section*{Data Availability}
The theoretical data used to support the findings of this study are already included in the article.
\section*{Authors Contributions}
The authors declare that the study was conducted in collaboration with the same responsibility. All authors read and approved the final manuscript.
\section*{Conflicts of interest}
On behalf of all authors, the corresponding author states that there is no conflict of interest.
\bibliographystyle{elsart-num}
\bibliography{bibl} 
\end{document}